\documentclass[a4paper,10pt]{article}

\usepackage{geometry}
\usepackage{epsfig}
\usepackage{amsmath}
\usepackage{amssymb}
\usepackage{amsthm}
\usepackage{mathrsfs}
\usepackage{color}
\usepackage{tikz-cd}
\usepackage{rotating}

\usepackage{amsthm}
\newtheorem{theorem}{Theorem}[section]
\newtheorem{definition}[theorem]{Definition}
\newtheorem{lemma}[theorem]{Lemma}
\newtheorem{proposition}[theorem]{Proposition}
\newtheorem{corollary}[theorem]{Corollary}
\newtheorem{remark}[theorem]{Remark}
\newtheorem{example}[theorem]{Example}

\usepackage{amsfonts}

\newcommand{\oo}{{\mathbb{O}}}
\newcommand{\hh}{{\mathbb{H}}}

\newcommand{\cc}{{\mathbb{C}}}
\newcommand{\rr}{{\mathbb{R}}}

\newcommand{\nn}{{\mathbb{N}}}

\newcommand{\J}{\mathbb{J}}

\newcommand{\s}{{\mathbb{S}}}
\renewcommand{\P}{{\mathbb{P}}}

\newcommand{\I}{\mathcal{I}}

\newcommand{\B}{\mathcal{B}}
\newcommand{\C}{\mathcal{C}}
\newcommand{\E}{\mathcal{E}}

\newcommand\vs[1]{{#1}_s^\circ}

\newcommand{\punto}{\bullet}

\newcommand\re{\operatorname{Re}}
\newcommand\im{\operatorname{Im}}

\newcommand{\ui}{\imath}

\newcommand{\OO}{\Omega}

\newcommand{\mr}{\mathrm}

\newcommand{\mc}{\mathcal}
\newcommand{\hslashslash}{%
  \raisebox{.9ex}{%
    \scalebox{.7}{%
      \rotatebox[origin=c]{18}{$-$}%
    }%
  }%
}
\newcommand{\fslash}{%
  {%
   \vphantom{f}%
   \ooalign{\kern.05em\smash{\hslashslash}\hidewidth\cr$f$\cr}%
   \kern.05em
  }%
}

\title{\bf Slice regular functions and orthogonal complex structures over $\rr^8$}

\author{Riccardo Ghiloni\\
 Alessandro Perotti\\
\small Dipartimento di Matematica, Universit\`a di Trento\\ 
\small Via Sommarive 14, I-38123 Povo Trento, Italy\\
\small riccardo.ghiloni@unitn.it, alessandro.perotti@unitn.it\\
\and
Caterina Stoppato
\\ 
\small Dipartimento di Matematica e Informatica ``U. Dini'', Universit\`a di Firenze \\
\small Viale Morgagni 67/A, I-50134 Firenze, Italy\\
\small caterina.stoppato@unifi.it}

\date{  }

\begin{document}

\maketitle


\begin{abstract}
This work looks at the theory of octonionic slice regular functions through the lens of differential topology. It proves a full-fledged version of the Open Mapping Theorem for octonionic slice regular functions. Moreover, it opens the path for a possible use of slice regular functions in the study of almost-complex structures in eight dimensions.
\end{abstract}

\vspace{.4cm}


\thanks{\small \noindent{\bf Acknowledgements.} This work was partly supported by GNSAGA of INdAM, by the INdAM project ``Hypercomplex function theory and applications'' and by the PRIN 2017 project ``Real and Complex Manifolds'' of the Italian Ministry of Education (MIUR). The third author is also supported by Finanziamento Premiale FOE 2014 ``Splines for accUrate NumeRics: adaptIve models for Simulation Environments'' of MIUR. The authors are grateful to the anonymous referee for the precious suggestions.}


\section{Introduction}\label{sec:introduction}

The theory of slice regular functions was introduced over quaternions in~\cite{cras, advances} and largely developed in the following years. The monograph~\cite{librospringer} and subsequent works describe its many resemblances to the theory of holomorphic functions of one complex variable, as well as new interesting phenomena due to the noncommutative setting.

The theory soon found useful applications to open problems in other areas of mathematics, including the problem of classifying Orthogonal Complex Structures (OCSs) on open dense subsets $\rr^4\setminus\Lambda$ of $\rr^4$. For the definition of OCS, see Section~\ref{sec:constantoacs}. The works~\cite{wood,viaclovsky} had provided a classification for the case when $\Lambda$ has Hausdorff dimension less than $1$ and the case when $\Lambda$ is a circle or a straight line. At those times, the only function class available for classification was the class of conformal maps between open subsets of $\rr^4$. In contrast with the case of $\rr^2$, by a famous theorem due to Liouville, this class consists only of quaternionic linear fractional transformations composed with reflections. This made even the case when $\Lambda$ is a parabola unapproachable. The work~\cite{ocs} significantly widened the panorama by making the class of injective slice regular functions available as a tool for classification. This required a detailed study of the differential topology of quaternionic slice regular functions.

In the present work, we look at the theory of \emph{octonionic} slice regular functions, introduced in~\cite{rocky} and briefly described in Section~\ref{sec:preliminaries}, through the lens of differential topology. This study has an independent interest, because of the peculiar features of the nonassociative setting of octonions. We obtain a full-fledged version of the Open Mapping Theorem for octonionic slice regular functions, after the partial results of~\cite{wang,gpsdivisionalgebras}. Moreover, we open the path for a possible use of slice regular functions in the study of almost-complex structures in eight dimensions.

The paper is structured as follows.

Section~\ref{sec:preliminaries} is devoted to preliminaries. It recalls the definition and properties of the algebra of octonions and of the class of octonionic slice regular functions. It reviews known properties of constant OCSs on $\rr^{2n}$, specializing them to the cases $n=1,2,4$ of the division algebras $\cc,\hh,\oo$ of complex numbers, quaternions and octonions. Then some instrumental results are proven, which play a crucial role throughout the paper. Finally, the standard Orthogonal Almost-Complex Structure $\J$ of $\oo\setminus\rr\simeq\cc^+\times S^6$ is presented.

Section~\ref{sec:differential} is a first study of the real differential and of the real Jacobian of octonionic slice regular functions. 

Section~\ref{sec:inducedacs} studies the possibility to induce, by pushing $\J$ forward through a slice regular function $f$ at a point $x_0$, an almost-complex structure on the tangent space at $f(x_0)$.

Section~\ref{sec:singularsets} studies the singular sets $N_f$ of slice regular functions $f$ and proves the Quasi-open Mapping Theorem for these functions.

Section~\ref{sec:fibers} studies the fibers of octonionic slice regular functions and proves the aforementioned Open Mapping Theorem.

In Section~\ref{sec:branchedcoverings}, the branch set of an octonionic slice regular function $f$ is proven to coincide with its singular set. This makes it possible to push $\J$ forward through any injective slice regular function $f:\OO\to\oo$ and induce an almost-complex structure on $f(\OO\setminus\rr)$.


\section{Preliminaries}\label{sec:preliminaries}

\subsection{The real algebra of octonions}

Let $\oo$ denote the $^*$-algebra of octonions, built by iterating the so-called Cayley-Dickson construction:
\begin{itemize}
\item $\cc=\rr+i\rr$, $(\alpha+i\beta)(\gamma+i\delta)=\alpha\gamma-\beta\delta+i(\alpha\delta+\beta\gamma)$, $(\alpha+i\beta)^c=\alpha-i\beta\ \forall\,\alpha,\beta,\gamma,\delta\in\rr$;
\item $\hh=\cc+j\cc$, $(\alpha+j\beta)(\gamma+j\delta)=\alpha\gamma-\beta^c\delta+j(\alpha^c\delta+\beta\gamma)$, $(\alpha+j\beta)^c=\alpha^c-j\beta\ \forall\,\alpha,\beta,\gamma,\delta\in\cc$;
\item $\oo=\hh+\ell\hh$, $(\alpha+\ell \beta)(\gamma+\ell \delta)=\alpha\gamma-\delta\beta^c+\ell(\alpha^c\delta+\gamma\beta)$, $(\alpha+\ell \beta)^c=\alpha^c-\ell \beta\ \forall\,\alpha,\beta,\gamma,\delta\in\hh$.
\end{itemize}
We will now quickly overview the properties of $\oo$, referring the reader to~\cite{baez,ebbinghaus,schafer,libroward} for more details.

$\oo$ is a non commutative and non associative unitary real algebra. The associative nucleus and the center of $\oo$ both coincide with the subalgebra generated by $1$, which is denoted simply by $\rr$. Although $\oo$ is not associative, it is \emph{alternative}: the associator $(x,y,z)=(xy)z-x(yz)$ of three elements vanishes whenever two of them coincide. Alternativity implies several properties, including the following.
\begin{itemize}
\item {[Moufang identities]} For all elements $a,x,y$ of an alternative algebra,
\begin{eqnarray}
(xax)y &=& x(a(xy))\,,\label{moufang1}\\
y(xax) &=& ((yx)a)x\,,\label{moufang2}\\
(xy)(ax) &=& x(ya) x\,.\label{moufang3}
\end{eqnarray}
\item {[Artin's Theorem]} In an alternative algebra, the subalgebra generated by any two elements is associative.
\item {[Power associativity]} For all $x$ in an alternative algebra, $(x,x,x)=0$, so that the expression $x^n$ can be written unambiguously for all $n \in \nn$.
\end{itemize}

$\oo$ is a \emph{$^*$-algebra} over $\rr$ because the map $x\mapsto x^c$ is a \emph{$^*$-involution}, i.e., an $\rr$-linear transformation with the following properties: $(x^c)^c=x$ and $(xy)^c=y^cx^c$ for every $x,y$; $x^c=x$ for every $x \in \rr$. We point out that $(r+v)^c=r-v$ for all $r \in \rr$ and all $v$ in the Euclidean orthogonal complement of $\rr$. The \emph{norm} function $n(x):=xx^c$ coincides with the squared Euclidean norm $\Vert x\Vert^2$  and $n(xy) = n(x) n(y)$ for all $x,y \in \oo$; the \emph{trace} function $t(x):=x+x^c$ has $t(xy^c)$ equal to twice the standard scalar product $\langle x,y\rangle$ of $\oo=\rr^8$.

$\oo$ is a division algebra because every nonzero element $x$ has a multiplicative inverse, namely $x^{-1} = n(x)^{-1} x^c = x^c\, n(x)^{-1}$. For all $x,y\in\oo$:
\begin{itemize}
\item if $x\neq0$ then $(x^{-1},x,y)=0$;
\item if $x,y\neq 0$ then $(xy)^{-1} = y^{-1}x^{-1}$.
\end{itemize}
Well-known results due to Frobenius and Zorn state that $\rr,\cc,\hh$ and $\oo$ are the only (finite-dimensional) alternative division algebras.

The set of octonionic \emph{imaginary units}
\begin{equation}
\s=\s_\oo:=\{x \in \oo\, | \, t(x)=0, n(x)=1\} = \{w \in \oo \, |\, w^2=-1\}\,,
\end{equation}
is a $6$-dimensional sphere. The $^*$-subalgebra generated by any $J \in \s$, i.e., $\cc_J:=\rr+J\rr$, is $^*$-isomorphic to the complex field $\cc$ (endowed with the standard multiplication and conjugation) through the $^*$-isomorphism 
\[\phi_J\ :\ \cc\to\cc_J\,,\quad \alpha+i\beta\mapsto\alpha+\beta J\,.\]
It holds that
\begin{equation} \label{eq:slice}
\oo=\text{$\bigcup_{J \in \s}\cc_J$}
\end{equation}
and $\cc_I \cap \cc_J=\rr$ for every $I,J \in \s$ with $I \neq \pm J$. As a consequence, every element $x$ of $\oo \setminus \rr$ can be written as follows: $x=\alpha+\beta J$, where $\alpha \in \rr$ is uniquely determined by $x$, while $\beta \in \rr$ and $J \in \s$ are uniquely determined by $x$, but only up to sign. If $x \in \rr$, then $\alpha=x$, $\beta=0$ and $J$ can be chosen arbitrarily in $\s$. Therefore, it makes sense to define the \emph{real part} $\re(x)$ and the \emph{imaginary part} $\im(x)$ by setting $\re(x):=t(x)/2=(x+x^c)/2 = \alpha$ and $\im(x):=x-\re(x)=(x-x^c)/2=\beta J$. It also makes sense to call the Euclidean norm $\Vert x\Vert=\sqrt{n(x)}=\sqrt{a^2+\beta^2}$ the \emph{modulus} of $x$ and to denote it as $|x|$.  The \emph{vector product} of two elements $v,w\in\im(\oo)$ is defined as $v\times w:=\im(vw)=\langle v,w\rangle+vw$. For all $x,y,z\in\oo$, the associator $(x,y,z)=(xy)z-x(yz)$ has the property $\re((x,y,z))=0$ (see \cite[page 188]{baez}). The algebra $\oo$ has the following useful property.
\begin{itemize}
\item {[Splitting property]} For each imaginary unit $J \in \s$, there exist $J_1,J_2,J_3 \in \oo$ such that $\{1,J,J_1,JJ_1,J_2,JJ_2,J_3,JJ_3\}$ is a real vector basis of $\oo$, called a \emph{splitting basis} of $\oo$ associated to $J$. We call a splitting basis $\{1,J,J_1,JJ_1,J_2,JJ_2,J_3,JJ_3\}$ of $\oo$ \emph{distinguished} if there exists a real $^*$-algebra isomorphism $\oo\to\oo$ mapping $1,J_1,J_2,J_3,J$ to $1,i,j,k,\ell$, respectively.
\end{itemize}
The existence, for each $J\in\s$, of a distinguished splitting basis $\{1,J,J_1,JJ_1,J_2,JJ_2,J_3,JJ_3\}$ of $\oo$ follows from~\cite[Propositions 2.5 and 2.6]{rocky}.

On the $8$-dimensional real vector space $\oo$, we consider the natural Euclidean topology and differential structure. The relative topology on each $\cc_J$ with $J \in \s$ clearly agrees with the topology determined by the natural identification between $\cc_J$ and $\cc$, i.e., $\phi_J$ is a homeomorphism. Given a subset $E$ of $\cc$, its \emph{circularization} $\OO_E$ is defined as the following subset of $\oo$:
\[
\OO_E:=\left\{x \in \oo \, \big| \, \exists \alpha,\beta \in \rr, \exists J \in \s \mathrm{\ s.t.\ } x=\alpha+\beta J, \alpha+i\beta \in E\right \}\,.
\]
A subset of $\oo$ is termed \emph{circular} if it equals $\OO_E$ for some $E\subseteq\cc$. For instance, given $x=\alpha+\beta J \in \oo$, we have that
\[
\s_x:=\alpha+\beta \, \s=\{\alpha+\beta I \in \oo \, | \, I \in \s\}
\]
is circular, as it is the circularization of the singleton $\{\alpha+i\beta\}\subseteq \cc$. We observe that $\s_x=\{x\}$ if $x \in \rr$. On the other hand, for $x \in \oo \setminus \rr$, the set $\s_x$ is obtained by real translation and dilation from the sphere $\s$. If $D$ is a non-empty subset of $\cc$ that is invariant under the complex conjugation $z=\alpha+i\beta \mapsto \overline{z}=\alpha-i\beta$, then for each $J \in \s$ the map $\phi_J$ naturally embeds $D$ into a ``slice'' of $\OO_D=\bigcup_{J\in\s}\phi_J(D)$, that is, $\phi_J(D)=\OO_D\cap\cc_J$.


\subsection{Slice regular functions}

We now overview the definition of slice regular function given in~\cite{perotti} and recall some useful properties of these functions. Consider the complexified $^*$-algebra $\oo_{\cc}=\oo \otimes_{\rr} \cc=\{x+\ui y \, | \, x,y \in \oo\}$, with
\[
(x+\ui y)(x'+\ui y'):=xx'-yy'+\ui (xy'+yx'),\quad (x+\ui y)^c := x^c+\ui y^c\,.
\]
We also set $\overline{x+\ui y}:=x-\ui y$. The center of $\oo_{\cc}$ is the real $^*$-subalgebra $\rr_\cc=\rr+\ui\rr$. If we identify $\cc$ with $\rr_\cc$ then, for all $J\in\s$, the previously defined map $\phi_J:\cc\to\cc_J$ extends to
\[\phi_J\ :\ \oo_\cc \to \oo,\quad x+\ui y\mapsto x+Jy\,.\]

\begin{definition} \label{def:slice-function}
Let $D$ be a non-empty subset of $\cc$ preserved by complex conjugation and consider its circularization $\OO_D$. A function $F:D\to\oo_\cc$ is a \emph{stem function} if $F(\overline{z})=\overline{F(z)}$ for every $z \in D$. A function $f:\OO_D \to \oo$ is a \emph{(left) slice function} if there exists a stem function $F:D \to \oo_\cc$ such that the diagram
\begin{equation}\label{lifting}
\begin{tikzcd}
D\arrow{d}{\phi_J}\arrow{r}{F}&\oo_\cc\arrow{d}{\phi_J}\\
\OO_D\arrow{r}{f}&\oo
\end{tikzcd}
\end{equation}
commutes for each $J\in\s$. In this situation, we say that $f$ is induced by $F$ and we write $f=\I(F)$. If $F$ is $\rr_\cc$-valued, then we say that the slice function $f$ is \emph{slice preserving}.  If $\OO:=\OO_D$, we denote by $\mc{S}(\OO)$ the set of all slice functions from $\OO$ to $\oo$; we denote by $\mc{S}_\rr(\OO)$ the subset of $\mc{S}(\OO)$ formed by all slice preserving functions.
\end{definition}

The term `slice preserving' is justified by the following property (cf.~\cite[Proposition~10]{perotti}): a stem function $F$ is $\rr_\cc$-valued if, and only if, the slice function $f=\I(F)$ maps every ``slice'' $\phi_J(D)$ into $\cc_J$. We point out that each slice function $f$ is induced by a unique stem function $F$.

The algebraic structure of slice functions can be described as follows, see \cite[\S 2]{gpsalgebra}.

\begin{proposition}\label{prop:algebra}
The set $\mr{Stem}(D,\oo_\cc)$ of all stem functions from $D$ to $\oo_\cc$ is an alternative $^*$-algebra over $\rr$ with pointwise addition $(F+G)(z) = F(z)+G(z)$, multiplication $(FG)(z) = F(z)G(z)$ and conjugation $F^c(z) = F(z)^c$. The center of this $^*$-algebra comprises all stem functions from $D$ to $\rr_\cc$. Let $\OO:=\OO_D$ and consider the mapping
\[\I\ : \mr{Stem}(D,\oo_\cc)\to\mc{S}(\OO)\,.\]
Besides the pointwise addition $(f,g) \mapsto f + g$, there exist unique operations of multiplication $(f,g) \mapsto f \cdot g$ and conjugation $f \mapsto f^c$ on $\mc{S}(\OO)$ so that the mapping $\I$ is a $^*$-algebra isomorphism. The center of this $^*$-algebra coincides with the $^*$-subalgebra $\mc{S}_\rr(\OO)$ of slice preserving functions.
\end{proposition}

If $f$ is slice preserving then $(f \cdot g)(x)=(g \cdot f)(x)=f(x)g(x)$ but neither equality holds in general. The pointwise product $x\mapsto f(x)g(x)$ is denoted as $fg$. The \emph{normal function} of $f$ in $\mc{S}(\OO)$, defined as
\[
N(f)=f \cdot f^c=\I(FF^c),
\]
is a slice preserving function. It coincides with $f^2$ if $f$ is slice preserving.

It is also useful to define the \emph{spherical value}, $\vs f:\OO \to \oo$, and the \emph{spherical derivative}, $f'_s:\OO \setminus \rr \to \oo$, of any $f\in\mc{S}(\OO)$ by setting
\[
\vs f(x):=\frac{1}{2}(f(x)+f(x^c))
\quad \text{and} \quad
f'_s(x):=\frac{1}{2}\im(x)^{-1}(f(x)-f(x^c)).
\] 
The original article~\cite{perotti} and subsequent papers used the notations $v_s f$ and $\partial_s f$, respectively, and remarked that $\vs  f \in \mc{S}(\OO),f'_s \in \mc{S}(\OO \setminus \rr)$. For all $x \in \OO\setminus\rr$ it holds that
\begin{equation}\label{eq:representation}
f(x)=\vs  f(x) + (\im \cdot\, f'_s)(x) = \vs  f(x)+ \im(x)  f'_s(x),
\end{equation}
where the second equality holds because $\im$ is a slice preserving element of $\mc{S}(\oo)$. Notice that $f$ is slice preserving if, and only if, $\vs f$ and $f'_s$ are real-valued.

Among octonionic slice functions, we consider a special class having nice properties that recall those of holomorphic functions of a complex variable. This class of functions was introduced in~\cite{rocky}, although we follow here the presentation of~\cite{perotti}. If $\OO = \OO_D$ is open, then for any $J \in \s$ the slice $\OO_J:=\OO\cap\cc_J=\phi_J(D)$ is open in the relative topology of $\cc_J$; therefore, $D$ itself is open. In this case, within the $^*$-algebra of stem functions we can consider the $^*$-subalgebras of continuous, continuously differentiable, real analytic and holomorphic stem functions $F:D\to \oo_\cc$, the last being defined by the condition $\frac{\partial F}{\partial\overline{z}}\equiv0$ with
\[
\frac{\partial F}{\partial\overline{z}}:=\frac{1}{2}\left(\frac{\partial F}{\partial\alpha}+\ui\frac{\partial F}{\partial\beta}\right).
\]
These four $^*$-subalgebras induce, through the $^*$-isomorphism $\I$, four $^*$-subalgebras of $\mc{S}(\OO)$ that we denote by $\mc{S}^0(\OO), \mc{S}^1(\OO), \mc{S}^\omega(\OO)$ and $\mc{SR}(\OO)$, respectively. A function $f:\OO \to \oo$ is called \emph{slice regular} if it belongs to $\mc{SR}(\OO)$. Equivalently, $f = \I(F) \in \mc{S}^1(\OO)$ is slice regular if 
\[\partial f/\partial x^c:=\I(\partial F/\partial \overline{z})\]
vanishes identically in $\OO$. The analogously defined function 
\[f'_c=\partial f/\partial x:=\I(\partial F/\partial z)\]
on $\OO$ is called the \emph{slice derivative} (or complex derivative) of $f$. The following Leibniz rules hold:
\begin{align}
&(f\cdot g)'_c=f'_c\cdot g+f\cdot g'_c\,,\label{eq:leibnizcullen}\\
&(f\cdot g)^\circ_s=f^\circ_s\cdot g^\circ_s+\im^2f'_s\cdot g'_s\,,\quad (f\cdot g)'_s=f'_s\cdot g^\circ_s+f^\circ_s\cdot g'_s\,,\label{eq:leibnizspherical}
\end{align}
where we point out that $\im^2(\alpha+\beta J)=-\beta^2$ for all $\alpha,\beta\in\rr$ and all $J\in\s$. For all $x\in\OO\setminus\rr$, it also holds:
\begin{align}
&(f^c)^\circ_s(x)=f^\circ_s(x)^c\,,\quad (f^c)'_s(x)=f'_s(x)^c\label{eq:conjugatespherical}\\
&(N(f))^\circ_s(x)=n(f^\circ_s(x))+\im(x)^2n(f'_s(x))\,,\quad (N(f))'_s(x)=t(f^\circ_s(x)f'_s(x)^c)\,,\label{eq:normalspherical}
\end{align}

Because we are now working with an open $D$, we can decompose it into a disjoint union of open subsets $D_t\subseteq\cc$, each of which either
\begin{enumerate}
\item intersects the real line $\rr$, is connected and preserved by complex conjugation; or
\item does not intersect $\rr$ and has two connected components $D^+_t,D^-_t$, switched by complex conjugation.
\end{enumerate}
In the former case, the resulting $\OO_{D_t}$ is called a \emph{slice domain} because each intersection $\OO_{D_t}\cap\cc_J$ with $J \in \s$ is a domain in the complex sense (more precisely, it is an open connected subset of $\cc_J$). In case 2, we will call $\OO_{D_t}$ a \emph{product domain} as it is homeomorphic to the topological product between the complex domain $D^+_t$ and the sphere $\s$. Thus, any mention of $\mc{SR}(\OO)$ will imply that $\OO$ is a disjoint union of slice domains and product domains within the algebra $\oo$.

Polynomials and convergent power series are examples of slice regular functions, see~\cite[Theorem 2.1]{rocky}.

\begin{proposition}
Every polynomial of the form $\sum_{m=0}^n x^ma_m = a_0+x a_1 + \ldots +x^n a_n$ with coefficients $a_0, \ldots, a_n \in \oo$ is a slice regular function on $\oo$. Every power series of the form $\sum_{n \in \nn} x^n a_n$ converges in a ball $B(0, R) = \{x \in \oo\, | \,\Vert x \Vert<R\}$. If $R>0$, then the sum of the series is a slice regular function on $B(0,R)$.
\end{proposition}

Actually, $\mc{SR}(B(0,R))$ coincides with the $^*$-algebra of power series converging in $B(0, R)$ with the operations
\[\left(\sum_{n \in \nn} x^n a_n\right)\cdot\left(\sum_{n \in \nn} x^n b_n\right) = \sum_{n \in \nn} x^n \sum_{k=0}^na_kb_{n-k}\,,\]
\[\left(\sum_{n \in \nn} x^n a_n\right)^c= \sum_{n \in \nn} x^n a_n^c\,.\]
This is a consequence of~\cite[Theorem 2.12]{rocky}. With the same operations, the polynomials over $\oo$ form a $^*$-subalgebra of the $^*$-algebra $\mc{SR}(\oo)$ of \emph{octonionic entire functions}. Any octonionic (convergent) power series or polynomial is slice preserving if, and only if, its coefficients are real.

\begin{example}\label{ex:Delta}
If we fix an octonion $y$, the binomial $f(x) := x-y$ is a slice regular function on $\oo$. The conjugate function is $f^c(x) = x-y^c$ and the normal function $N(f)(x) = (x-y) \cdot (x-y^c) = x^2-x(y+y^c)+yy^c$ coincides with the slice preserving quadratic polynomial
\[
\Delta_y(x):=x^2-xt(y)+n(y)\,.
\]
If $y' \in\oo$, then $\Delta_{y'}=\Delta_y$ if and only if $\s_{y'}=\s_y$.
\end{example}

The next result, concerning the zero set $V(f)=\{x\in\OO \, | \, f(x)=0\}$ of $f$, was proven for power series in~\cite{ghiloni} and extended to all $f\in\mc{S}(\OO)$ in~\cite{gpsdivisionalgebras}.

\begin{theorem}\label{zerosonsphere}
If $f \in \mc{S}(\OO)$, then for every $x \in \OO$ the sets $\s_x \cap V(f)$ and $\s_x \cap V(f^c)$ are both empty, both singletons, or both equal to $\s_x$. Moreover,
 \[
V(N(f))=\bigcup_{x \in V(f)}\s_x=\bigcup_{x \in V(f^c)}\s_x\,.
 \]
Finally, for all $g \in \mc{S}(\OO)$,
\[\bigcup_{x\in V(f\cdot g)}\s_x =\bigcup_{x\in V(f)\cup V(g)}\s_x.\]
\end{theorem}

\begin{example}\label{ex:binomial}
If $f(x) := x-y$, whence $f^c(x) =x-y^c$ and $N(f)=\Delta_y(x):=x^2-xt(y)+n(y)$, then 
\[V(f)=\{y\}\,,\quad V(f^c)=\{y^c\}\,,\quad V(N(f))=\s_y\,.\] For all constant functions $g\equiv c$, we have $(f\cdot g)(x)= xc-yc$, whence $V(f\cdot g)$ is $\{y\}$ when $c \neq 0$ and it is $\oo$ when $c=0$.
\end{example}

The same works studied in greater detail $V(f\cdot g)$ for $f,g\in\mc{S}(\OO)$. For slice regular functions, zeros can be factored out as follows (see~\cite[Theorem 22]{perotti}).

\begin{theorem}\label{thm:factorization}
Let $f:\OO\to\oo$ be a slice regular function and let $y\in\OO$. The zero set $V(f)$ includes $y$ if, and only if, there exists $g\in\mc{SR}(\OO)$ such that
\[f(x)=(x-y)\cdot g(x)\,.\]
The zero set $V(f)$ contains $\s_y$ if, and only if, there exists $h\in\mc{SR}(\OO)$ such that
\[f(x)=\Delta_y(x)\cdot h(x)=\Delta_y(x) h(x)\,.\]
As a consequence, $y\in V(f)$ if, and only if, $\Delta_{y}$ divides $N(f)$.
\end{theorem}

If $f\not\equiv0$ then, by~\cite[Remark 11.2]{gpssingularities}, there exists a largest $m\in\nn$ such that $\Delta_{y}^m$ divides $f$. The even number $2m$ is called the \emph{spherical multiplicity} of $f$ at $\s_{y}$, if $y\in\OO\setminus\rr$. If, instead, $y\in\OO\cap\rr$, there exists a largest $n\in\nn$ such that $(x-y)^n$ divides $f(x)$; the number $n$ is called the \emph{classical multiplicity} of $f$ at $y$. If $f\equiv0$, all spherical and classical multiplicities of $f$ are set to $+\infty$.

If $N(f)\not\equiv0$ then, for all $y\in\OO$, the largest $m$ such that $\Delta_{y}^m$ divides $N(f)$ is called the \emph{total multiplicity} of $f$ at $\s_{y}$. If $N(f)\equiv0$, we may set to $+\infty$ the total multiplicities of $f$ at all $y\in\OO$.

We conclude this section by recalling two properties of the zeros of octonionic slice regular functions, namely,~\cite[Theorem 3.5]{gpsdivisionalgebras} and~\cite[Proposition 3.7]{gpsdivisionalgebras}. We use the notations $\cc_J^+:=\{\alpha+\beta J\, | \,\alpha,\beta\in\rr,\beta>0\}$ and $\OO_{J}^+:=\OO\cap\cc_J^+$.

\begin{theorem}\label{structureofzeros}
Assume that $\OO$ is a slice domain or a product domain and let $f \in \mc{SR}(\OO)$.
\begin{itemize}
\item If $f\not\equiv0$ then the intersection $V(f)\cap \cc_J^+$ is closed and discrete in $\OO_J$ for all $J\in\s$ with at most one exception $J_0$, for which it holds that $f_{|_{\OO_{J_0}^+}} \equiv 0$.
\item If, moreover, $N(f)\not\equiv0$ then $V(f)$ is a union of isolated points or isolated spheres $\s_x$.
\end{itemize}
\end{theorem}

\begin{proposition}\label{SRnonsingular}
Let $f\in\mc{SR}(\OO)$. The equality $N(f)\equiv0$ implies $f\equiv0$ if, and only, $\OO$ is a union of slice domains.
\end{proposition}

The work~\cite{perotti} provided the following example of slice regular function $f$ over a product domain with $f\not\equiv0$ but $N(f)\equiv0$.

\begin{example}\label{ex:1fin}
Fix $J_0\in\s$. We define $\eta_{J_0}:\oo\setminus\rr\to\oo$ by the formula
\[\eta_{J_0}(x):=\frac12+\frac{\im(x)}{|\im(x)|}\frac{J_0}{2}\,.\]
Then $\eta_{J_0}$ is slice regular in $\oo\setminus\rr$ and its zero set $V(\eta_{J_0})$ is $\cc_{J_0}^+$. By direct computation, $\eta_{J_0}^c=\eta_{-J_0}$ and $N(\eta_{J_0})=\eta_{J_0}\cdot\eta_{-J_0}\equiv0$.
\end{example}

More properties of octonionic slice functions and slice regular functions can be found in~\cite{wang,gpsdivisionalgebras}.


\subsection{Orthogonal almost-complex structures over the octonions}\label{sec:constantoacs}

In this work, we will thoroughly use constant OCSs on $\oo=\rr^8$, according to the next definition. We denote the standard scalar product of $\rr^{2n}$ by the symbol $\langle\,,\rangle$. As usual, $M^\B_{\B'}(F)$ denotes the matrix associated to any linear map $F$ with respect to a real vector basis $\B$ of its domain and a basis $\B'$ of its codomain.

\begin{definition}
A \emph{constant complex structure} on $\rr^{2n}$ is an $\rr$-linear endomorphism $\mc{J}_0:\rr^{2n}\to\rr^{2n}$ such that $(\mc{J}_0\circ\mc{J}_0)(u)=-u$ for all $u\in\rr^{2n}$. If, moreover, $\langle\mc{J}_0(u),\mc{J}_0(v)\rangle=\langle u,v\rangle$ for all $u,v\in\rr^{2n}$, then $\mc{J}_0$ is termed \emph{orthogonal}.
\end{definition}

\begin{definition}
An \emph{orthogonal almost-complex structure (OACS)} on a Riemannian manifold $(M,g)$ is an almost-complex structure $\mc{J}$ on $M$ such that $g_{x_0}(\mc{J}_{x_0}(u),\mc{J}_{x_0}(v))=g_{x_0}(u,v)$ for all $x_0\in M$ and all $u,v\in T_{x_0}M$. If $\mc{J}$ is also integrable, then it is called an \emph{orthogonal complex structure (OCS)}.
\end{definition}

The interested reader may find a friendly introduction to OCSs in~\cite{salamonumi}. The same article uses the notation $Z_n$ for the Hermitian symmetric space $\frac{SO(2n)}{U(n)}$ of all constant orthogonal complex structures on $\rr^{2n}$ that induce the standard orientation. The lower-dimensional cases are described as follows: $Z_1$ is a singleton because $SO(2)=U(1)$; there exist isomorphisms
\begin{align*}
&\varphi_2:\cc\P^1\to Z_2\,,\\
&\varphi_3:\cc\P^3\to Z_3\,,\\
&\varphi_4:\mathcal{Q}^6\phantom{I}\to Z_4\,,
\end{align*}
where $\mathcal{Q}^6$ is a quadric in $\cc\P^7$. The chart $\cc\to Z_2$, $\xi\mapsto\varphi_2[1:\xi]$ is portrayed in Table 1 and the chart $\cc^3\to Z_3$, $(\xi_1,\xi_2,\xi_3)\mapsto\varphi_3[1:\xi_1:\xi_2:\xi_3]$ in Table 2, in accordance with~\cite[Proposition 3.4]{borisov}. As explained in~\cite[page 129]{borisov}, the orientation induced on $\rr^{2n}$ by a constant OCS $\mc{J}_0$ is determined by the Pfaffian, $\mr{Pf}(M)$, of the matrix $M=M^\E_\E(\mc{J}_0)$ associated to $\mc{J}_0$ with respect to the standard basis $\E=\E_{2n}$. The determinant is, instead, $\det(M)=1$ irrespective of orientation. For the definition and basic properties of the Pfaffian, see~\cite[\S8.5]{librocohn1}. According to the convention adopted in the latter reference, $\mc{J}_0$ induces a positive orientation on $\rr^{2n}$ if, and only if, $\mr{Pf}(M)=-1$, while a negative orientation corresponds to $\mr{Pf}(M)=1$. By Formula (7) in~\cite[\S8.5]{librocohn1}, an orthogonal basis $\B$ of $\rr^{2n}$ is positively oriented if, and only if, the Pfaffian of $M^\B_\B(\mc{J}_0)$ equals the Pfaffian of $M^\E_\E(\mc{J}_0)$.

\begin{sidewaystable}
\centering \small
\[\left(\begin{array}{cccc}
0&-1+|\xi|^2&-2\beta&2\alpha\\
1-|\xi|^2&0&2\alpha&2\beta\\
2\beta&-2\alpha&0&-1+|\xi|^2\\
-2\alpha&-2\beta&1-|\xi|^2&0
\end{array}
\right)\]
\caption{The matrix associated to $\varphi_2[1:\xi]$ with respect to the standard basis of $\rr^4$. Here, $\xi=\alpha+i\beta$.}

\[\left(\begin{array}{cccccc}
0&-1+|\xi_1|^2+|\xi_2|^2-|\xi_3|^2&-2(\beta_1-\alpha_3 \beta_2+\alpha_2 \beta_3)&2 (\alpha_1+\alpha_2 \alpha_3+\beta_2 \beta_3)&-2 (\alpha_3 \beta_1+\beta_2-\alpha_1 \beta_3)&-2 (-\alpha_2+\alpha_1 \alpha_3+\beta_1 \beta_3)\\
1-|\xi_1|^2-|\xi_2|^2+|\xi_3|^2&0&2 (\alpha_1-\alpha_2 \alpha_3-\beta_2 \beta_3)&2 (\beta_1+\alpha_3 \beta_2-\alpha_2 \beta_3)&2 (\alpha_2+\alpha_1 \alpha_3+\beta_1 \beta_3)&2 (-\alpha_3 \beta_1+\beta_2+\alpha_1 \beta_3)\\
2 (\beta_1-\alpha_3 \beta_2+\alpha_2 \beta_3)&2 (-\alpha_1+\alpha_2 \alpha_3+\beta_2 \beta_3)&0&-1+|\xi_1|^2-|\xi_2|^2+|\xi_3|^2&-2 (-\alpha_2 \beta_1+\alpha_1 \beta_2+\beta_3)&2 (\alpha_1 \alpha_2+\alpha_3+\beta_1 \beta_2)\\
-2 (\alpha_1+\alpha_2 \alpha_3+\beta_2 \beta_3)&-2 (\beta_1+\alpha_3 \beta_2-\alpha_2 \beta_3)&1-|\xi_1|^2+|\xi_2|^2-|\xi_3|^2&0&-2 (\alpha_1 \alpha_2-\alpha_3+\beta_1 \beta_2)&2 (\alpha_2 \beta_1-\alpha_1 \beta_2+\beta_3)\\
2 (\alpha_3 \beta_1+\beta_2-\alpha_1 \beta_3)&-2 (\alpha_2+\alpha_1 \alpha_3+\beta_1 \beta_3)&2 (-\alpha_2 \beta_1+\alpha_1 \beta_2+\beta_3)&2 (\alpha_1 \alpha_2-\alpha_3+\beta_1\beta_2)&0&-1-|\xi_1|^2+|\xi_2|^2+|\xi_3|^2\\
2 (-\alpha_2+\alpha_1 \alpha_3+\beta_1 \beta_3)&-2 (-\alpha_3 \beta_1+\beta_2+\alpha_1 \beta_3)&-2 (\alpha_1 \alpha_2+\alpha_3+\beta_1 \beta_2)&-2 (\alpha_2 \beta_1-\alpha_1\beta_2+\beta_3)&1+|\xi_1|^2-|\xi_2|^2-|\xi_3|^2&0
\end{array}
\right)\]
\caption{The matrix associated to $\varphi_3[1:\xi_1:\xi_2:\xi_3]$ with respect to the standard basis of $\rr^6$. Here, $\xi_t=\alpha_t+i\beta_t$.}
\end{sidewaystable}

Using the division algebras $\cc,\hh,\oo$, the spaces $Z_1,Z_2,Z_4$ can also be described as follows.

\begin{itemize}
\item In $\cc=\rr^2$ the standard orientation is induced by multiplication by $i$. With respect to the standard basis $\E_2=\{1,i\}$, the associated matrix is
\[H:=\left(\begin{array}{rr}
0 & -1\\
1 & 0
\end{array}
\right)\,.\]
It holds that $\mr{Pf}(H)=-1$ according to~\cite[\S8.5]{librocohn1}. There are no other constant OCSs inducing the same orientation. The matrix associated to multiplication by $-i$ with respect to the standard basis $\{1,i\}$ is $-H$, which has $\mr{Pf}(-H)=1=-\mr{Pf}(H)$.
\item In $\hh=\rr^4$, the standard orientation is induced by left multiplication by $i$, so that the standard basis $\E_4=\{1,i,j,k\}$ (with $k=ij$) is positively oriented. Each constant OCS $\mc{J}_0$ on $\hh$ can be identified with a point $I$ of the $2$-sphere $\s_\hh$ of unit elements of $\im(\hh)$, or \emph{quaternionic imaginary units}, as follows. Since $\mc{J}_0$ is orthogonal, it holds that $\mc{J}_0(1)=I\in\s_\hh$. On the vector space orthogonal to $1$ and $I$ in $\hh\simeq\rr^4$, say the span of $J,IJ$ for some $J\in\s_\hh$ orthogonal to $I$, there is only one constant OCS compatible with the given orientation. Thus, $\mc{J}_0$ is left multiplication by a fixed $I\in\s_\hh$ and the matrix associated to $\mc{J}_0$ with respect to the basis $\{1,I,J,IJ\}$ is
\[\left(\begin{array}{cc}
H& 0 \\
0 & H
\end{array}
\right)\,,\]
whose Pfaffian is $-1$ according to~\cite[\S8.5]{librocohn1}.
\item In $\oo=\rr^8$, the standard orientation is induced by left multiplication by $i$. It can be checked by direct computation that the standard basis $\E_8=\{1,i,j,k,\ell,\ell i, \ell j, \ell k\}$ is positively oriented. Now consider a constant OCS $\mc{J}_0$ on $\oo$. As before, $\mc{J}_0$ maps $1$ to an octonionic imaginary unit $J\in\s$. If $\B=\{1,J,v_1,\ldots,v_6\}$ is any positively oriented orthogonal basis starting with $1$ and $J$, then the matrix $M_\B^\B(\mc{J}_0)$ associated to $\mc{J}_0$ with respect to $\B$ has the form
\[\left(\begin{array}{cc}
H& 0 \\
0 & S
\end{array}
\right)\,,\]
where $S$ is the $6\times6$ real matrix associated to an arbitrary constant OCS (compatible with the chosen orientation) on the $6$-dimensional vector space $\cc_J^\perp$, with respect to the basis $\{v_1,\ldots,v_6\}$.
\end{itemize}

It is clear from the previous list that, for $n=2,4$, the space $Z_n$ is a fiber bundle over $S^{2n-2}$ with fiber $Z_{n-1}$. This statement is also true for $n=3$: if we fix a vector $v_0\in\rr^6$, choosing a constant OCS $\mc{J}_0$ on $\rr^6$ means choosing a value $\mc{J}_0(v_0)=w_0$ in the unit $4$-sphere of $\{v_0\}^\perp$ and then choosing a constant OCS on the $4$-dimensional vector space $\{v_0,w_0\}^\perp$, compatible with the chosen orientation. For instance: by taking a second look at Table 2, it is not hard to spot the matrix associated to $\varphi_2[1:\xi]$ in the upper left $4\times4$ minor of the matrix associated to $\varphi_3[1:\xi:0:0]$.

In the previous list, when we considered $\hh$, we only mentioned left multiplication by a constant $I\in\s_\hh$. Right multiplication by such an $I$ induces a negative orientation. Indeed, with respect to a positively oriented orthogonal basis $\{1,I,J,IJ\}$, the associated matrix is
$\left(\begin{array}{cc}
H& 0 \\
0 & -H
\end{array}
\right)$,
whose Pfaffian is $1$.

We now focus on the case of octonions $\oo$, treated in~\cite{bryant}. The next result subsumes some results of~\cite[\S4.6, page 205]{libroward} on octonionic rotations and a statement from~\cite[page 190]{bryant}: namely, that left multiplications by octonionic imaginary units form a subclass of $Z_4$, while right multiplications do not. Both for the sake of completeness and because our choice of orientation is different from the choice of~\cite{bryant}, we include a proof of this result.

\begin{theorem}\label{thm:LandR}
For each $p\in\oo$, let us set $L_p(x):=px$ and $R_p(x):=xp$ for all $x\in\oo$.
\begin{enumerate}
\item $L_p,R_p$ are transformations of $\oo$ if, and only if, $p\neq0$. In this case, they are orientation-preserving conformal transformations with scaling factor $|p|$ and $L_{p}^{-1}=L_{p^{-1}}, R_{p}^{-1}=R_{p^{-1}}$.
\item $L_p,R_p$ are (special) orthogonal transformations of $\oo$ if, and only if, $|p|=1$.
\item $L_p,R_p$ are (constant) orthogonal complex structures on $\oo$ if, and only if, $p\in\s$.
\item For every imaginary unit $J\in\s$, $L_J$ induces a positive orientation of $\oo$, while $R_J$ induces a negative orientation.
\end{enumerate}
\end{theorem}

\begin{proof}
By direct computation,
\begin{align*}
&\re(L_p(x)^cL_p(y))=\re((x^cp^c)(py))=\re((x^c,p^c,py))+\re(x^c(n(p)y))=|p|^2\re(x^cy)\,,\\
&\re(R_p(x)R_p(y)^c)=\re((xp)(p^cy^c))=\re((x,p,p^cy^c))+\re(x(n(p)y^c))=|p|^2\re(xy^c)\,.
\end{align*}
Thus, $\langle L_p(x),L_p(y)\rangle=|p|^2\langle x,y\rangle$ and $\langle R_p(x),R_p(y)\rangle=|p|^2\langle x,y\rangle$. By direct computation, $\det(L_p)=|p|^8=\det(R_p)$. The first two statements immediately follow.

The third statement follows from the fact that $L_p\circ L_p=L_{p^2}$ and $R_p\circ R_p=R_{p^2}$ equal $-id_\oo$ if, and only if, $p^2=-1$.

We can prove the fourth statement as follows. Consider the standard basis $\E=\E_8=\{1,i,j,k,\ell,\ell i, \ell j, \ell k\}$. The map $\s\to\{\pm1\}\,,\ J\mapsto\mr{Pf}\big(M^\E_\E(L_J)\big)$ is continuous, whence constant. So is the map $\s\to\{\pm1\}\,,\ J\mapsto\mr{Pf}\big(M^\E_\E(R_J)\big)$. Moreover, for $J=i$, we can explicitly compute:
\[M_\E^\E(L_i)=\left(\begin{array}{rrrr}
H&0&0&0\\
0&H&0&0\\
0&0&-H&0\\
0&0&0&-H
\end{array}
\right)\,,\quad M_\E^\E(R_i)=\left(\begin{array}{rrrr}
H&0&0&0\\
0&-H&0&0\\
0&0&H&0\\
0&0&0&H
\end{array}
\right)\,.\]
The former matrix has Pfaffian $-1$, while the latter has Pfaffian $1$.
\end{proof}

As a byproduct, we can make the following remark.
\begin{remark}
Every distinguished splitting basis $\B=\{1,J,J_1,JJ_1,J_2,JJ_2,J_3,JJ_3\}$ is a positively oriented orthogonal basis. Indeed, by direct computation, $M^\B_\B(L_J)$ is
\[\left(\begin{array}{rrrr}
H&0&0&0\\
0&H&0&0\\
0&0&H&0\\
0&0&0&H
\end{array}
\right)\,,\]
and $\mr{Pf}\big(M^\B_\B(L_J)\big)=-1=\mr{Pf}\big(M^\E_\E(L_J)\big)$.
\end{remark}

Theorem~\ref{thm:LandR} showed that $\mathcal{L}:=\{L_J\}_{J\in\s}$ is a $6$-dimensional subset of the $12$-dimensional space $Z_4$. Further examples of elements of $Z_4$ can be found by conjugation with the conformal transformations $L_p$ or $R_p$ (in contrast with the quaternionic case, when such conjugations produce other left multiplications by imaginary units). This is done in the next two propositions. The two propositions and the subsequent lemma are probably well-known to experts, but we have not been able to find these specific results explicitly stated in the literature.

\begin{proposition}\label{prop:compositionL}
For all $J\in\s$ and all $p\in\oo\setminus\{0\}$, the composed transformation
\[(L_p\circ L_J\circ L_{p^{-1}})(x)=p(J(p^{-1}x))\]
belongs to $Z_4$. If $p\in\cc_J$, then $L_p\circ L_J\circ L_{p^{-1}}=L_J$. If $p\in\im(\oo)$, then $L_p\circ L_J\circ L_{p^{-1}}=L_{J'}$ with $J'=pJp^{-1}\in\s$. For all $p\in\oo\setminus(\cc_J\cup\im(\oo))$, it holds that $L_p\circ L_J\circ L_{p^{-1}}\neq L_{J'}$ for all $J'\in\s$.
\end{proposition}

\begin{proof}
Let $p\in\oo\setminus\{0\}$. To prove that $L_p\circ L_J\circ L_{p^{-1}}$ is a constant complex structure, it suffices to observe that
\[(L_p\circ L_J\circ L_{p^{-1}})\circ(L_p\circ L_J\circ L_{p^{-1}})=L_p\circ L_J\circ L_J\circ L_{p^{-1}}=-L_p\circ L_{p^{-1}}=-id_\oo\,,\]
where the first equality follows from Artin's theorem. Since $L_p$ is an orientation-preserving conformal transformation with scaling factor $|p|$, it follows at once that $L_p\circ L_J\circ L_{p^{-1}}$ is a constant OCS that induces the standard orientation.
 
Now let us prove the last statements. If $p\in\cc_J$, then $p,J,p^{-1}$ associate and commute, whence $L_p\circ L_J\circ L_{p^{-1}}=L_J$. Now suppose $p\in\oo\setminus\cc_J$, and let $V$ denote the $4$-dimensional associative subalgebra of $\oo$ generated by $J$ and $p$. For all $v\in V$, it holds that
\[(L_p\circ L_J\circ L_{p^{-1}})(v)=p(J(p^{-1}v))=(pJp^{-1})v\]
by Artin's theorem. For all $w\in V^\perp$, it holds that
\[(L_p\circ L_J\circ L_{p^{-1}})(w)=p(J(p^{-1}w))=(p^{-1}Jp)w\,.\]
To prove the last equality, we may assume without loss of generality $V=\hh$ and $w=\ell q$ for some $q\in\hh$. It holds that
\[p(J(p^{-1}(\ell q)))=p(J(\ell ((p^{-1})^cq)))=p(\ell(J^c(p^{-1})^cq))=\ell(p^cJ^c(p^{-1})^cq)=(p^{-1}Jp)(\ell q)\,.\]
To conclude, we observe:
\[pJp^{-1}=p^{-1}Jp\Leftrightarrow p^2J=Jp^2\Leftrightarrow p^2\in\cc_J\Leftrightarrow p\in\im(\oo)\,,\]
where the last equivalence follows from the assumption $p\in\oo\setminus\cc_J$.
\end{proof}

The subset of $Z_4$ described in Proposition~\ref{prop:compositionL} admits an alternative representation, described in the next result.

\begin{proposition}\label{prop:compositionR}
For all $J\in\s$ and all $p\in\oo\setminus\{0\}$, the composed transformation
\[(R_p\circ L_J\circ R_{p^{-1}})(x)=(J(xp^{-1}))p\]
belongs to $Z_4$. If $p\in\cc_J$, then $R_p\circ L_J\circ R_{p^{-1}}=L_J$. Now assume $p\in\oo\setminus\cc_J$: it holds that $R_p\circ L_J\circ R_{p^{-1}}\neq L_{J'}$ for all $J'\in\s$; moreover, $R_p\circ L_J\circ R_{p^{-1}}=L_{s^{-1}}\circ L_{sJs^{-1}}\circ L_s$, where $\pm s$ are the solutions of $x^2=p$.
\end{proposition}

\begin{proof}
As in the previous Corollary, $R_p\circ L_J\circ R_{p^{-1}}$ is a constant OCS that induces the standard orientation.

If $p\in\cc_J$, then the algebra generated by $J,p$ and any $x\in\oo$ is associative, whence $R_p\circ L_J\circ R_{p^{-1}}=L_J$. Now suppose $p\in\oo\setminus\cc_J$ and let $V$ denote the $4$-dimensional associative subalgebra of $\oo$ generated by $J$ and $p$. For all $v\in V$, it holds that
\[(R_p\circ L_J\circ R_{p^{-1}})(v)=(J(vp^{-1}))p=Jv\]
by Artin's theorem. For all $w\in V^\perp$, it holds that
\[(R_p\circ L_J\circ R_{p^{-1}})(w)=(J(wp^{-1}))p=(pJp^{-1})w\,.\]
To prove the last equality, we may assume without loss of generality $V=\hh$ and $w=\ell q$ for some $q\in\hh$. It holds that
\[(J((\ell q)p^{-1}))p=(J(\ell(p^{-1}q)))p=(\ell(J^cp^{-1}q))p=\ell(pJ^cp^{-1}q)=((p^{-1})^cJp^c)(\ell q)=(pJp^{-1})(\ell q)\,.\]

We observe that the units $J$ and $pJp^{-1}\in\s$ are distinct because we have assumed $p\in\oo\setminus\cc_J$. As a consequence, $R_p\circ L_J\circ R_{p^{-1}}\neq L_{J'}$ for all $J'\in\s$.

The assumption $p\in\oo\setminus\cc_J$ also implies that the equation $x^2=p$ has exactly two solutions $\pm s$, not belonging to $\rr$ and included in the $2$-dimensional subalgebra generated by $1$ and $p$. This implies that the $4$-dimensional associative subalgebra of $\oo$ generated by $J':=sJs^{-1}$ and by $s^{-1}$ equals $V$. An inspection of the proof of Proposition~\ref{prop:compositionL} reveals that $L_{s^{-1}}\circ L_{J'}\circ L_{s}$ acts as $L_{s^{-1}J's}$ on $V$ and as $L_{sJ's^{-1}}$ on $V^\perp$. The equalities $s^{-1}J's= s^{-1}(sJs^{-1})s=J$ and $sJ's^{-1}=s(sJs^{-1})s^{-1}=s^2Js^{-2}=pJp^{-1}$  imply that $L_{s^{-1}}\circ L_{J'}\circ L_{s}$ acts as $L_J$ on $V$ and as $L_{pJp^{-1}}$ on $V^\perp$, whence $L_{s^{-1}}\circ L_{J'}\circ L_{s}=R_p\circ L_J\circ R_{p^{-1}}$.
\end{proof} 

We point out that $\oo\setminus\{0\}$ is not a multiplicative group but only a loop and that conjugation with the conformal transformation $L_p$ (for $p\in\oo\setminus\{0\}$) is no analog of an action on $Z_4$. Indeed, $L_q\circ(L_p\circ L_J\circ L_{p^{-1}})\circ L_{q^{-1}}$ may be different from $L_{qp}\circ L_J\circ L_{(qp)^{-1}}$.

\begin{example}
The value
\[(L_\ell\circ(L_i\circ L_j\circ L_{i^{-1}})\circ L_{\ell^{-1}})(1) = \ell(i(j(i^{-1}\ell^{-1})))=\ell(i(j(\ell i)^{-1}))\]
is different from the value
\[(L_{\ell i}\circ L_j\circ L_{(\ell i)^{-1}})(1)=(\ell i)(j(\ell i)^{-1}).\]
Indeed, $j(\ell i)^{-1}=-j(\ell i)=\ell(j i)=-\ell k$ and it holds that
\[-\ell(i(\ell k))=\ell(\ell (ik))=j\neq -j=-ki=-(\ell i)(\ell k).\]
\end{example}

Similar considerations apply to conjugation with $R_p$. Thus, further examples of constant OCS on $\oo$ can be produced.

Proposition~\ref{prop:compositionR} allows us to establish the next lemma, which will prove extremely useful throughout the paper.

\begin{lemma}\label{lem:scalarproduct}
Let $a,b\in\oo$ and let $J\in\s$. For all $v\in\cc_J$ it holds
\[\langle ab,v\rangle=\langle (Ja)b,Jv\rangle,\quad\langle (Ja)b,v\rangle=\langle J(ab),v\rangle\,.\]
Moreover, the orthogonal projection $\pi:(\cc_Ja)\,b\to\cc_J$ fulfills the equality $\pi((Ja)b)=J\pi(ab)$. As a consequence, $\pi$ is surjective if, and only if, $ab\not\in\cc_J^\perp$. Finally, $(\cc_Ja)\,b=\cc_J$ if, and only if, $ab\in\cc_J\setminus\{0\}$.
\end{lemma}

\begin{proof}
The thesis is obvious if $b=0$. Thus, we may assume $b\neq0$. It holds that
\[(Ja)b=R_{b}(L_J(a))=R_{b}(L_J(R_{b^{-1}}(ab)))\]
and
\[Jv=R_{b}((Jv)b^{-1})=R_{b}(J(vb^{-1}))=R_{b}(L_J(R_{b^{-1}}(v)))\,,\]
where we have applied Artin's theorem. Thus,
\[\langle (Ja)b,Jv\rangle=\big\langle (R_b\circ L_J\circ R_{b^{-1}})(ab),(R_b\circ L_J\circ R_{b^{-1}})(v)\big\rangle=\langle ab,v\rangle\,,\]
where the last equality is a consequence of Proposition~\ref{prop:compositionR}. Now, $\langle (Ja)b,v\rangle=\langle(-1)(ab),Jv\rangle=-\langle J(ab),-v\rangle=\langle J(ab),v\rangle$, where the first and second equalities are repeated applications of the formula we already proved.

We can prove the second statement as follows. Let $v:=\pi(ab)$ and $w:=\pi((Ja)b)$. Combining the equalities we have just proven with Artin's theorem, we obtain: 
\begin{align*}
&\langle w,1\rangle=\langle(Ja)b,1\rangle=-\langle ab,J\rangle=-\langle v,J\rangle\,,\\
&\langle w,J\rangle=\langle (Ja)b,J\rangle=\langle ab,1\rangle=\langle v,1\rangle\,.
\end{align*}
This proves that $w=Jv$, as desired.

The third statement follows from the second one by observing that the image of $\pi$ is the span of the mutually orthogonal vectors $v$ and $w=Jv$. This image is $\cc_J$ if, and only if, $v\neq0$, which is equivalent to $ab\not\in\cc_J^\perp$.

Finally, we can prove the fourth statement as follows. The equality $(\cc_Ja)\,b=\cc_J$ implies $ab\in\cc_J\setminus\{0\}$: indeed, it implies $(1a)b\in\cc_J$ (because $1\in\cc_J$) and it implies $ab\neq0$ (by dimensional considerations). Conversely, suppose $ab\in\cc_J\setminus\{0\}$: then $\pi((Ja)b)=J(ab)$, whence $|(Ja)b-J(ab)|^2=|(Ja)b|^2-|J(ab)|^2=0$ and $(Ja)b=J(ab)\in\cc_J\setminus\{0\}$; thus, $(\cc_J a)\,b$, which is the span of $ab,(Ja)b$, equals $\cc_J$.
\end{proof}

A relevant example of a nonconstant OACS can be constructed on $\oo\setminus\rr$.

\begin{definition}\label{def:standardstructure}
The standard almost-complex structure $\J$ on $\oo\setminus\rr$ is defined by setting
\[\J_{x_0}(v):=\frac{\im(x_0)}{|\im(x_0)|}\,v\]
for all $x_0\in\oo\setminus\rr$ and for all $v\in T_{x_0}(\oo\setminus\rr)=\oo$.
\end{definition}

\begin{remark}
If $x_0\in\cc_J^+\setminus\rr$ then $\J_{x_0}:T_{x_0}(\oo\setminus\rr)\to T_{x_0}(\oo\setminus\rr)$ is the same as $L_J:\oo\to\oo$. As a consequence, the almost-complex structure $\J$ on $\oo\setminus\rr$ is orthogonal.
\end{remark}

Endowing $\oo\setminus\rr$ with $\J$ is equivalent to considering the decomposition
\[\oo\setminus\rr=\rr+\rr^+\s\simeq\cc^+\times S^6\,,\]
where $S^6$ is endowed with its standard almost-complex structure, introduced in the celebrated work~\cite{kirchhoff} and proven to be non integrable in~\cite{eckmannfrolicher,ehresmannlibermann}.


\section{Real differentials of octonionic slice regular functions}\label{sec:differential}

This section is devoted to a first study of the real differential and of the real Jacobian of octonionic slice regular functions. This study will be applied in Section~\ref{sec:inducedacs} and further refined in Section~\ref{sec:singularsets}.

Every slice regular function $f:\OO\to\oo$ is automatically real analytic. The same is true for its spherical value $\vs f:\OO\to\oo$ and for its spherical derivative $f'_s:\OO\setminus\rr\to\oo$. The interested reader may find a proof in~\cite[Proposition 7]{perotti}. In the next proposition, and in the rest of the paper, any mention of $J$ will automatically imply that $J\in\s$.

\begin{proposition}\label{prop:differential}
Let $f:\OO\to\oo$ be a slice function.
\begin{enumerate}
\item If $x_0\in\OO\cap\rr$ and $f$ is slice regular, then
\[df_{x_0}(v) = v\,f'_c(x_0)\]
for all $v\in T_{x_0}\OO=\oo$.
\item If $x_0\in\OO\setminus\rr$ and $x_0\in\cc_J$, let us split $T_{x_0}\OO=\oo$ as $\cc_J\oplus\cc_J^\perp$. For all $w$ in the $6$-dimensional vector space $\cc_J^\perp$, the partial derivative of $f$ in the $w$ direction at $x_0$ exists and it equals $wf'_s(x_0)$. If, moreover, $f$ is slice regular then
\[df_{x_0}(v+w) = v\,f'_c(x_0) + w\,f'_s(x_0)\]
for all $v \in \cc_J,w\in\cc_J^\perp$.
\end{enumerate}
\end{proposition}

\begin{proof}
First suppose $x_0\in\OO\cap\rr$: the thesis follows from the definition of slice regular function.

Now suppose $x_0\in\OO\setminus\rr$. Let us consider Formula~\eqref{eq:representation} and take into account the fact that $\vs  f,f'_s$ are constant in $\s_{x_0}$. If $w\in T_{x_0}\s_{x_0}=\cc_J^\perp$, we immediately conclude that the partial derivative of $f$ in the $w$ direction at $x_0$ exists and equals $wf'_s(x_0)$. If, moreover, $f$ is slice regular, then for each $v \in \cc_J$ the partial derivative of $f$ in the $v$ direction at $x_0$ equals $vf'_c(x_0)$.
\end{proof}

Now let us look at the range of the differential.

\begin{theorem}\label{thm:rankofdifferential}
Let $f:\OO\to\oo$ be a slice regular function.
\begin{enumerate}
\item If $x_0\in\OO\cap\rr$, then the range of $df_{x_0}$ is either $\{0\}$ or $\oo$, depending on whether or not $f'_c(x_0)=0$. As a consequence, $df_{x_0}$ is invertible if, and only if, $f'_c(x_0)\neq0$. In such a case, $df_{x_0}$ is an orientation-preserving conformal transformation.
\item If we fix $x_0\in\OO_J\setminus\rr$ then the range of $df_{x_0}$ is the sum $V+W$ with $V:=\cc_J\,f'_c(x_0), W:=\cc_J^\perp\,f'_s(x_0)$. Moreover:
\begin{enumerate}
\item $V+W=\{0\}$ if $f'_s(x_0)=0=f'_c(x_0)$;
\item $V+W=V$ is a $2$-dimensional vector space if $f'_s(x_0)=0\neq f'_c(x_0)$;
\item $V+W=W$ is a $6$-dimensional vector space if $f'_s(x_0)\neq0$ and $f'_c(x_0)f'_s(x_0)^{-1}\in\cc_J^\perp$;
\item $V\oplus W=\oo$ if $f'_s(x_0)\neq0$ and $f'_c(x_0)f'_s(x_0)^{-1}\not\in\cc_J^\perp$. The direct sum is orthogonal if, and only if, $f'_c(x_0)f'_s(x_0)^{-1}\in\cc_J\setminus\{0\}$.
\end{enumerate}
As a consequence, $df_{x_0}$ is invertible if, and only if, $f'_s(x_0)\neq0$ and $f'_c(x_0)f'_s(x_0)^{-1}\not\in\cc_J^\perp$. It is an orientation-preserving conformal transformation if, and only if, $f'_s(x_0)\neq0$, $f'_c(x_0)f'_s(x_0)^{-1}\in\cc_J$ and $|f'_c(x_0)f'_s(x_0)^{-1}|=1$.
\end{enumerate}
\end{theorem}

\begin{proof}
The first statements immediately follow from Proposition~\ref{prop:differential}. As for the others, let $a:=f'_c(x_0)$ and $b:=f'_s(x_0)$, so that $V=\cc_Ja,W=\cc_J^\perp b$. 
\begin{enumerate}
\item[$(a)$] If $a=0=b$ then $V=W=\{0\}$.
\item[$(b)$] If $a\neq0=b$ then $V$ is a $2$-dimensional vector space and $W=\{0\}$.
\item[$(c)$] Suppose $b\neq0$ and $ab^{-1}\in\cc_J^\perp$. By Lemma~\ref{lem:scalarproduct}, $(Ja)b^{-1}$ belongs to $\cc_J^\perp$, too. Thus,
\[Vb^{-1}+Wb^{-1}=(\cc_Ja)\,b^{-1}+\cc_J^\perp = \rr\,ab^{-1} + \rr\,(Ja)b^{-1} + \cc_J^\perp = \cc_J^\perp\,,\]
whence $V+W$ is the $6$-dimensional vector space $\cc^\perp_Jb=W$.
\item[$(d)$] Suppose $b\neq0$ and $ab^{-1}\not\in\cc_J^\perp$. By Lemma~\ref{lem:scalarproduct}, the orthogonal projection $\pi:(\cc_Ja)\,b^{-1}\to\cc_J$ is surjective. Thus, 
\[Vb^{-1}+Wb^{-1}=(\cc_Ja)\,b^{-1}+\cc_J^\perp =\cc_J\oplus\cc_J^\perp=\oo\,.\]
Moreover, the direct sum $V\oplus W$ is orthogonal if, and only if, $Vb^{-1}=(\cc_Ja)\,b^{-1}$ equals $\cc_J$. By Lemma~\ref{lem:scalarproduct}, this happens if, and only if, $ab^{-1}\in\cc_J\setminus\{0\}$. The thesis now follows from the fact that $R_b$ is an orientation-preserving conformal transformation of $\oo$ mapping $Vb^{-1}$ to $V$ and $Wb^{-1}$ to $W$.\qedhere
\end{enumerate}
\end{proof}

\begin{corollary}\label{cor:singularset}
Let $f:\OO\to\oo$ be a slice regular function. Consider its \emph{singular set}
\[N_f:=\{x_0\in\OO \, | \, df_{x_0}\mathrm{\ is\ not\ invertible}\}
\]
and its \emph{degenerate set} $D_f:=V(f'_s)$. Then $N_f$ includes $D_f$ and
\begin{equation}\label{eq:rank6}
N_f\setminus D_f=\{x_0\in\OO\cap\rr \, | \, f'_c(x_0)=0\}\cup\bigcup_{J\in\s}\{x_0\in\OO_J\setminus\rr \, | \, f'_s(x_0)\neq 0, f'_c(x_0)f'_s(x_0)^{-1}\in\cc_J^\perp\}\,.
\end{equation}
\end{corollary}

The previous results can be reread in appropriate coordinates. We begin with a useful technical lemma.

\begin{lemma}\label{lemma:jacobian}
Let $\B=\{1,J,J_1,JJ_1,J_2,JJ_2,J_3,JJ_3\}$ be a distinguished splitting basis of $\oo$ and let us denote the real components of any vector $v$ with respect to $\B$ as $v_0,\ldots,v_7$. Fix octonions $a,b$ with $b\neq0,ab^{-1}\not\in\cc_J^\perp$ and set
\[\B(a,b):=\{a,Ja,J_1b,(JJ_1)b,J_2b,(JJ_2)b,J_3b,(JJ_3)b\}\,.\]
Then $\B(a,b)$ is a basis of $\oo$ and
\[M^{\B(a,b)}_\B(id)=\left(
\begin{array}{rrrrrrrr}
a_0 & -a_1 & -b_2 & -b_3 & -b_4 & -b_5 & -b_6 & -b_7\\
a_1 &  a_0 & b_3 & -b_2 & b_5 & -b_4 & b_7 & -b_6\\
a_2 & -a_3 & b_0 & b_1 & b_6 & -b_7 & -b_4 & b_5\\
a_3 &  a_2 & -b_1 & b_0 & -b_7 & -b_6 & b_5 & b_4\\
a_4 & -a_5 & -b_6 & b_7 & b_0 & b_1 & b_2 & -b_3\\
a_5 &  a_4 & b_7 & b_6 & -b_1 & b_0 & -b_3 & -b_2\\
a_6 & -a_7 & b_4 & -b_5 & -b_2 & b_3 & b_0 & b_1\\
a_7 &  a_6 & -b_5 & -b_4 & b_3 & b_2 & -b_1 & b_0
\end{array}
\right)\,.\]
Moreover,
\[\det M^{\B(a,b)}_\B(id) = |b|^4(\re^2(ab^{c})+\re^2(J(ab^{c}))).\]
\end{lemma}

\begin{proof}
The first formula is proven by direct computation, using the fact that $\B$ is a distinguished splitting basis and the following equalities:
\begin{align*}
a&=\sum_{t=0}^3(a_{2t}+a_{2t+1}J)J_t\,,\\
Ja&=\sum_{t=0}^3(-a_{2t+1}+a_{2t}J)J_t\,,\\
b&=\sum_{t=0}^3(b_{2t}+b_{2t+1}J)J_t\,,
\end{align*}
where $J_0:=1$. To prove the second formula, let us introduce a third basis of $\oo$. We define
\[\C:=\{ab^{-1},(Ja)b^{-1},J_1,JJ_1,J_2,JJ_2,J_3,JJ_3\}\,.\]
By Lemma \ref{lem:scalarproduct}, $\C$ is a basis of $\oo$. We notice that $\B(a,b)=R_b(\C)$, where $R_b:\oo\to\oo$ is the right multiplication by $b$. Thus, $\B(a,b)$ is a basis of $\oo$ and
\[M^{\B(a,b)}_\B(id)=M^{\C}_\B(id)\,M^{\B(a,b)}_\C(id)=M^{\C}_\B(id)\,M^{\C}_\C(R_b)\,.\]
By Theorem~\ref{thm:LandR}, we conclude that $\det M^{\C}_\C(R_b)=|b|^8$. Moreover,
\[M^{\C}_\B(id)=\left(
\begin{array}{rrrrrrrr}
(ab^{-1})_0 & ((Ja)b^{-1})_0 & 0 & 0 & 0 & 0 & 0 & 0\\
(ab^{-1})_1 & ((Ja)b^{-1})_1 & 0 & 0 & 0 & 0 & 0 & 0\\
(ab^{-1})_2 & ((Ja)b^{-1})_2 & 1 & 0 & 0 & 0 & 0 & 0\\
(ab^{-1})_3 & ((Ja)b^{-1})_3 & 0 & 1 & 0 & 0 & 0 & 0\\\
(ab^{-1})_4 & ((Ja)b^{-1})_4 & 0 & 0 & 1 & 0 & 0 & 0\\
(ab^{-1})_5 & ((Ja)b^{-1})_5 & 0 & 0 & 0 & 1 & 0 & 0\\\
(ab^{-1})_6 & ((Ja)b^{-1})_6 & 0 & 0 & 0 & 0 & 1 & 0\\
(ab^{-1})_7 & ((Ja)b^{-1})_7 & 0 & 0 & 0 & 0 & 0 & 1
\end{array}
\right)\,,\]
where, according to Lemma~\ref{lem:scalarproduct},

\begin{align*}
(ab^{-1})_0 &=\langle ab^{-1},1\rangle=\re(ab^{-1})\,,\\
((Ja)b^{-1})_1 &=\langle(Ja)b^{-1},J\rangle=\langle ab^{-1},1\rangle=\re(ab^{-1})\,,\\
(ab^{-1})_1&=\langle ab^{-1},J\rangle= -\re(J(ab^{-1}))\,,\\
((Ja)b^{-1})_0&=\langle(Ja)b^{-1},1\rangle=-\langle ab^{-1},J\rangle=\re(J(ab^{-1}))\,.
\end{align*}
As a consequence,
\begin{align*}
\det M^{\B(a,b)}_\B(id)&= \det M^{\C}_\B(id)\,\det M^{\C}_\C(R_b)\\
&=\left(\re^2(ab^{-1})+\re^2(J(ab^{-1}))\right)|b|^8\\
&=|b|^4\left(\re^2(ab^{c})+\re^2(J(ab^{c}))\right)\,.\qedhere
\end{align*}
\end{proof}

We are now ready to prove the next result.

\begin{theorem}\label{thm:jacobian}
Let $f:\OO\to\oo$ be a slice regular function. Fix $x_0\in\OO_J$ and a distinguished splitting basis $\B=\{1,J,J_1,JJ_1,J_2,JJ_2,J_3,JJ_3\}$ of $\oo$.
\begin{enumerate}
\item Suppose $x_0\not\in\rr$. If $f'_s(x_0)\neq0$ and $f'_c(x_0)f'_s(x_0)^{-1}\not\in\cc_J^\perp$ then
\[M^\B_\B(df_{x_0})=M^{\B(f'_c(x_0),f'_s(x_0))}_\B(id)\]
and
\[\det(df_{x_0})=|f'_s(x_0)|^4(\re^2(f'_c(x_0)f'_s(x_0)^{c})+\re^2(J(f'_c(x_0)f'_s(x_0)^{c})))\,.\]
The last equality also holds true when either $f'_s(x_0)=0$ or $f'_c(x_0)f'_s(x_0)^{-1}\in\cc_J^\perp$, in which cases both hands of the equality vanish.
\item Suppose $x_0\in\rr$. If $f'_c(x_0)\neq0$ then
\[M^\B_\B(df_{x_0})=M^{\B(f'_c(x_0),f'_c(x_0))}_\B(id)\]
and
\[\det(df_{x_0})=|f'_c(x_0)|^8\,.\]
The last equality also holds true when $f'_c(x_0)=0$, in which case both hands of the equality vanish.
\end{enumerate}
\end{theorem}

\begin{proof}
Assume $x_0\not\in\rr$ and let $a:=f'_c(x_0)$ and $b:=f'_s(x_0)$.

First suppose $b\neq 0, ab^{-1}\not\in\cc_J^\perp$. By Proposition~\ref{prop:differential}, $M^\B_{\B(a,b)}(df_{x_0})$ is the $8\times8$ identity matrix. Thus,
\[M^\B_\B(df_{x_0})=M^{\B(a,b)}_\B(id)\,M^\B_{\B(a,b)}(df_{x_0})=M^{\B(a,b)}_\B(id)\,,\]
which is the first formula in the statement. The second formula now follows from Lemma~\ref{lemma:jacobian}.

Now suppose either $b=0$ or $ab^{-1}\in\cc_J^\perp$. In both cases, clearly $|b|^4(\re^2(ab^{c})+\re^2(J(ab^{c})))=0$. Moreover, in both cases $\det(df_{x_0})=0$ by Theorem~\ref{thm:rankofdifferential}.

The case $x_0\in\rr$ can be treated similarly.
\end{proof}

\begin{corollary}\label{cor:system}
Let $f:\OO\to\oo$ be a slice regular function. If $x_0\in\OO_J\setminus\rr$, then $\det(df_{x_0})=0$ if, and only if, $\langle f'_c(x_0)f'_s(x_0)^{c},1\rangle=0$ and $\langle f'_c(x_0)f'_s(x_0)^{c},J\rangle=0$.
\end{corollary}


\section{Almost-complex structures induced by octonionic slice regular functions}\label{sec:inducedacs}

The results of Section~\ref{sec:constantoacs} allow us to explore whether slice regular functions are holomorphic with respect to appropriate almost-complex structures. In this section, we study holomorphy on the tangent space at a single point $x_0$. We will further this study in Section~\ref{sec:branchedcoverings}.

\begin{theorem}\label{thm:inducedacs}
Let $f:\OO\to\oo$ be a slice regular function and fix $x_0\in\OO_J\setminus\rr$. If $f'_s(x_0)\neq0$ and $f'_c(x_0)f'_s(x_0)^{-1}\not\in\cc_J^\perp$, then setting
\[\mc{J}_0(p+q):= Jp + (J(qf'_s(x_0)^{-1}))f'_s(x_0)\]
for all $p\in\cc_J\,f'_c(x_0)$ and all $q\in\cc_J^\perp\,f'_s(x_0)$ defines a constant complex structure on 
\[\oo=\cc_J\,f'_c(x_0)\oplus\cc_J^\perp\,f'_s(x_0)\]
such that 
\[\mc{J}_0\circ df_{x_0}=df_{x_0}\circ\J_{x_0}\,.\]
The structure $\mc{J}_0$ coincides with the constant orthogonal complex structure $R_{f'_s(x_0)}\circ L_J\circ R_{f'_s(x_0)^{-1}}$ if, and only if, $(x_0,f'_c(x_0),f'_s(x_0))=0$. It coincides with the constant orthogonal complex structure $L_J$ if, and only if, $f'_s(x_0)\in\cc_J$.
\end{theorem}

\begin{proof}
Let $a:=f'_c(x_0)$ and $b:=f'_s(x_0)$, so that $df_{x_0}(v+w)=va+wb$. We already proved that $b\neq 0, ab^{-1}\not\in\cc_J^\perp$ imply $\cc_Ja\oplus\cc_J^\perp b=\oo$ (the sum being direct, though not necessarily orthogonal). Thus, setting $\mc{J}_0(p+q)=Jp+(J(qb^{-1}))b$ for all $p\in\cc_J\,a$ and all $q\in\cc_J^\perp\,b$ leads to a well-defined endomorphism of $\oo$.

For all $v\in\cc_J$ and all $w\in\cc_J^\perp$, we have
\[(\mc{J}_0\circ df_{x_0})(v+w)=\mc{J}_0(va+wb)=J(va)+ (J(wbb^{-1}))b=(Jv)a+(Jw)b\]
where we have taken into account Artin's theorem. This formula coincides with
\[(df_{x_0}\circ\J_{x_0})(v+w)=df_{x_0}(Jv+Jw)=(Jv)a+(Jw)b\,.\]
The fact that $\mc{J}_0=df_{x_0}\circ\J_{x_0}\circ df_{x_0}^{-1}$, where $\J_{x_0}$ is a complex structure on $\oo$, immediately implies that $\mc{J}_0$ is a complex structure on $\oo$.

Finally, the equality $\mc{J}_0(p+q)=(J(pb^{-1}))b+(J(qb^{-1}))b$ is equivalent to $Jp=(J(pb^{-1}))b$. This happens for all $p\in\cc_Ja$  if, and only if, $(J,a,b)=0$. On the other hand, the equality $\mc{J}_0(p+q)=Jp+Jq$ is equivalent to $(J(qb^{-1}))b=Jq$. We claim that this happens for all $q\in\cc_J^\perp b$ if, and only if, $b\in\cc_J$.

To prove our claim, without lost of generality we can assume $J=\ell$ and prove the following assertion: the equality $(\ell w)b=\ell(wb)$ holds for all $w\in\hh$ if, and only if, $b\in\cc_\ell$. Suppose $b=b_1+\ell b_2$ with $b_1,b_2\in\hh$. Then
\begin{align*}
&(\ell w)b=(\ell w)(b_1+\ell b_2)=-b_2w^c+\ell(b_1w)\,,\\
&\ell(wb)=\ell (w(b_1+\ell b_2))=\ell (wb_1+\ell(w^cb_2))=-w^cb_2+\ell(wb_1)\,.
\end{align*}
It holds that $b_2w^c=w^cb_2$ and $b_1w=wb_1$ for all $w\in\hh$ if, and only if, $b_1,b_2\in\rr$. We immediately derive our thesis: $(\ell w)b=\ell(wb)$ holds for all $w\in\hh$ if, and only if, $b\in\cc_\ell$.
\end{proof}

When the push-forward of $\J$ via $f_{|_{\OO\setminus\rr}}$ is well-defined, it coincides at each point $f(x_0)$ with the structure defined in Theorem~\ref{thm:inducedacs}. We will determine when the push-forward is well-defined in Section~\ref{sec:branchedcoverings}. We can characterize orthogonality as follows.

\begin{theorem}\label{thm:inducedoacs}
Let $f:\OO\to\oo$ be a slice regular function and fix $x_0\in\OO_J\setminus\rr$. Suppose $f'_s(x_0)\neq0$ and $f'_c(x_0)f'_s(x_0)^{-1}\not\in\cc_J^\perp$. The structure $\mc{J}_0$ defined in Theorem~\ref{thm:inducedacs} is orthogonal if, and only if, $(x_0,f'_c(x_0),f'_s(x_0))=0$.
\end{theorem}

\begin{proof}
Let $a:=f'_c(x_0)$ and $b:=f'_s(x_0)$, so that $\mc{J}_0(p+q)=Jp+(J(qb^{-1}))b$ for all $p\in\cc_J\,a$ and all $q\in\cc_J^\perp\,b$. According to Theorem~\ref{thm:inducedacs}, we have to prove that $\mc{J}_0$ is orthogonal if, and only if, it coincides with $R_{b}\circ L_J\circ R_{b^{-1}}$. By Theorem~\ref{thm:LandR}, this is the same as proving that the transformation $F:=R_{b^{-1}}\circ\mc{J}_0\circ R_b$ is orthogonal if, and only if, $F=L_J$. We note that, for all $u\in(\cc_Ja)b^{-1}$ and all $w\in\cc_J^\perp$,
\[F(u+w)=(J(ub))b^{-1}+Jw\,.\]
If we fix any distinguished splitting basis 
\[\B=\{1,J,J_1,JJ_1,J_2,JJ_2,J_3,JJ_3\}\]
of $\oo$ and we consider the basis 
\[\C:=\{ab^{-1},(Ja)b^{-1},J_1,JJ_1,J_2,JJ_2,J_3,JJ_3\}\,,\]
then we have
\[M^\C_\C(F)=\left(\begin{array}{cccc}
H & 0 & 0 & 0\\
0 & H & 0 & 0\\
0 & 0 & H & 0\\
0 & 0 & 0 & H
\end{array}
\right)\,,\]
where $H:=\left(\begin{array}{rr}
0 & -1\\
1 & 0
\end{array}
\right)$. We already computed
\[M^\C_\B(id)=
\left(\begin{array}{cccc}
A & 0 & 0 & 0\\
B & I_2 & 0 & 0\\
C & 0 & I_2 & 0\\
D & 0 & 0 & I_2
\end{array}
\right),\quad
\left(\begin{array}{cccc}
A\\
B\\
C\\
D
\end{array}
\right):=\left(\begin{array}{rr}
(ab^{-1})_0 & ((Ja)b^{-1})_0\\
(ab^{-1})_1 & ((Ja)b^{-1})_1\\
(ab^{-1})_2 & ((Ja)b^{-1})_2\\
(ab^{-1})_3 & ((Ja)b^{-1})_3\\\
(ab^{-1})_4 & ((Ja)b^{-1})_4\\
(ab^{-1})_5 & ((Ja)b^{-1})_5\\\
(ab^{-1})_6 & ((Ja)b^{-1})_6\\
(ab^{-1})_7 & ((Ja)b^{-1})_7
\end{array}
\right)\,.\]
Since we already proved that
\[A=\left(\begin{array}{cc}
\re(ab^{-1}) & \re(J(ab^{-1}))\\
-\re(J(ab^{-1})) & \re(ab^{-1})
\end{array}
\right)\,,\]
the matrix $A$ is invertible and it commutes with $H$. By direct computation,
\begin{align*}
M^\B_\B(F)&=M^\C_\B(id)\,M^\C_\C(F)M^\B_\C(id)\\
&=
\left(\begin{array}{cccc}
A & 0 & 0 & 0\\
B & I_2 & 0 & 0\\
C & 0 & I_2 & 0\\
D & 0 & 0 & I_2
\end{array}
\right)
\left(\begin{array}{cccc}
H & 0 & 0 & 0\\
0 & H & 0 & 0\\
0 & 0 & H & 0\\
0 & 0 & 0 & H
\end{array}
\right)
\left(\begin{array}{cccc}
A^{-1} & 0 & 0 & 0\\
-BA^{-1} & I_2 & 0 & 0\\
-CA^{-1} & 0 & I_2 & 0\\
-DA^{-1} & 0 & 0 & I_2
\end{array}
\right)\\
&=\left(\begin{array}{cccc}
AH & 0 & 0 & 0\\
BH & H & 0 & 0\\
CH & 0 & H & 0\\
DH & 0 & 0 & H
\end{array}
\right)
\left(\begin{array}{cccc}
A^{-1} & 0 & 0 & 0\\
-BA^{-1} & I_2 & 0 & 0\\
-CA^{-1} & 0 & I_2 & 0\\
-DA^{-1} & 0 & 0 & I_2
\end{array}
\right)\\
&=\left(\begin{array}{cccc}
H & 0 & 0 & 0\\
(BH-HB)A^{-1} & H & 0 & 0\\
(CH-HC)A^{-1} & 0 & H & 0\\
(DH-HD)A^{-1} & 0 & 0 & H
\end{array}
\right)\,.
\end{align*}
Visibly, $M^\B_\B(F)$ is an orthogonal matrix if, and only if, it coincides with
\[\left(\begin{array}{cccc}
H & 0 & 0 & 0\\
0 & H & 0 & 0\\
0 & 0 & H & 0\\
0 & 0 & 0 & H
\end{array}
\right)=M^\B_\B(L_J)\,,\]
as desired.
\end{proof}

The necessary and sufficient condition for orthogonality is fulfilled by a nontrivial class of functions.

\begin{remark}\label{rmk:oneslicepreserving}
Let $f:\OO\to\oo$ be a slice regular function. If there exists $J\in\s$ such that $f(\OO_J)\subseteq\cc_J$, then $(x_0,f'_c(x_0),f'_s(x_0))=0$ for all $x_0\in\OO\setminus\rr$. Indeed, under this hypothesis, for all $x_0\in\OO$ the spherical derivative $f'_s(x_0)$ belongs to $\cc_J$ and the slice derivative $f'_c(x_0)$ belongs to the associative subalgebra of $\oo$ generated by $x_0$ and $\cc_J$.
\end{remark}


\section{Singular sets and quasi-openness}\label{sec:singularsets}

This section studies the singular set $N_f$ of any slice regular function $f$ and proves the Quasi-open Mapping Theorem for slice regular functions.

In Corollary~\ref{cor:singularset}, we saw that $N_f$ includes the degenerate set $D_f$ and we determined $N_f\setminus D_f$ by means of equality~\eqref{eq:rank6}. By definition, $N_f$ is a real analytic subset of $\OO$. We recall that, for any open $U\subseteq\oo$, a \emph{real analytic subset of $U$} is a set of the form $\psi^{-1}(0)$ for some real analytic function $\psi:U\to\rr$. In particular, a real analytic subset of $U$ is a closed subset of $U$. After recalling the next definition and stating a technical lemma, we can add the subsequent properties of $N_f$.

\begin{definition}
A slice regular function $f$ on $\OO$ is called \emph{slice constant} if $f_{|_{\OO_J}}$ is locally constant for each $J\in\s$. The subset of $\mc{SR}(\OO)$ that comprises slice constant functions is denoted by $\mc{SC}(\OO)$.
\end{definition}

\begin{lemma}\label{lem:partialderivatives}
Let $f:\OO\to\oo$ be a slice regular function and fix $I\in\s$. It holds that
\begin{align}\label{eq:partialderivatives}
&\frac{\partial \vs f(\alpha+\beta I)}{\partial \alpha}=\vs{(f'_c)}(\alpha+\beta I)\,,\\
&\frac{\partial f'_s(\alpha+\beta I)}{\partial \alpha}=(f'_c)'_s(\alpha+\beta I)\,,\notag\\
&\frac{\partial \vs f(\alpha+\beta I)}{\partial \beta}=-\beta(f'_c)'_s(\alpha+\beta I)\,,\notag\\
&\frac{\partial f'_s(\alpha+\beta I)}{\partial \beta}=-\beta^{-1}f'_s(\alpha+\beta I)+\beta^{-1}\vs{(f'_c)}(\alpha+\beta I)\,.\notag
\end{align}
As a consequence,
\begin{align}\label{eq:partialderivativesofsphericalvalue}
&\frac{\partial \vs f(\alpha+\beta I)}{\partial \alpha} = f'_s(\alpha+\beta I)+\beta \frac{\partial f'_s(\alpha+\beta I)}{\partial \beta}\,,\\
&\frac{\partial \vs f(\alpha+\beta I)}{\partial \beta} =-\beta \frac{\partial f'_s(\alpha+\beta I)}{\partial \alpha}\,.\notag
\end{align}
\end{lemma}

\begin{proof}
Using the definition of $\vs f(\alpha+\beta I),f'_s(\alpha+\beta I)$ and the equality $\frac{\partial f(\alpha+\beta I)}{\partial \alpha}=-I\frac{\partial f(\alpha+\beta I)}{\partial \beta}=f'_c(\alpha+\beta I)$, we can prove our first statement by direct computation:
\begin{align*}
\frac{\partial \vs f(\alpha+\beta I)}{\partial \alpha}&=\frac12 \left(f'_c(\alpha+\beta I)+f'_c(\alpha-\beta I)\right)=\vs{(f'_c)}(\alpha+\beta I)\,,\\
\frac{\partial f'_s(\alpha+\beta I)}{\partial \alpha}&=(2\beta I)^{-1}\left(f'_c(\alpha+\beta I)-f'_c(\alpha-\beta I)\right)=(f'_c)'_s(\alpha+\beta I)\,,\\
\frac{\partial \vs f(\alpha+\beta I)}{\partial \beta}&=\frac12 \left(If'_c(\alpha+\beta I)-If'_c(\alpha-\beta I)\right)=-\beta(f'_c)'_s(\alpha+\beta I)\,,\\
\frac{\partial f'_s(\alpha+\beta I)}{\partial \beta}&=-(2\beta^2 I)^{-1}\left(f(\alpha+\beta I)-f(\alpha-\beta I)\right)+(2\beta I)^{-1}\left(If'_c(\alpha+\beta I)+If'_c(\alpha-\beta I)\right)\\
&=-\beta^{-1}f'_s(\alpha+\beta I)+\beta^{-1}\vs{(f'_c)}(\alpha+\beta I)\,.
\end{align*}
Our second statement immediately follows.
\end{proof}

\begin{proposition}\label{prop:singularset}
Let $\OO$ be either a slice domain or a product domain and let $f:\OO\to\oo$ be a slice regular function.
\begin{enumerate}
\item If $f$ is constant, then $N_f=V(f'_c)=\OO$ and $D_f=\OO\setminus\rr$.
\item If $f$ is slice constant but $f$ is not constant, then $N_f=V(f'_c)=\OO$ while $D_f$ is a circular proper real analytic subset of $\OO\setminus\rr$.
\item If $f$ is not slice constant, then $N_f$ and $V(f'_c)$ are proper real analytic subsets of $\OO$ and $D_f$ is a circular proper real analytic subset of $\OO\setminus\rr$.
\end{enumerate}
In particular, $N_f$ has dimension $8$ if, and only if, $N_f=\OO$, which in turn happens if, and only if, $f$ is slice constant.
\end{proposition}

\begin{proof}
Since $f'_s:\OO\setminus\rr\to\oo$ is a real analytic function, constant on each $6$-sphere of the form $\s_{x_0}$, either $D_f=V(f'_s)$ is a circular proper real analytic subset of $\OO\setminus\rr$ or $D_f=\OO\setminus\rr$. In the latter case, $f$ coincides with $\vs f$ throughout $\OO$ and Formulas~\eqref{eq:partialderivativesofsphericalvalue} imply that $\vs f$ is constant in $\OO\setminus\rr$. It follows that $f$ is constant in $\OO$.

Now let us consider $V(f'_c)$. By the definition of $f'_c$, its zero set equals $\OO$ if, and only if, $f$ is slice constant. Otherwise, it is a proper real analytic subset of $\OO$.

We are left with proving that if $N_f=\OO$ then $f'_c$ or $f'_s$ vanish identically. This can be argued using different techniques for slice domains and product domains.
\begin{itemize}
\item Suppose $\OO$ is a slice domain and $N_f=\OO$. Then, by Theorem~\ref{thm:jacobian}, $f'_c(x_0)=0$ for all $x_0\in\OO\cap\rr$. As a consequence, $f'_c$ vanishes identically in $\OO$.
\item Suppose $\OO$ is a product domain and $N_f=\OO$. According to Corollary~\ref{cor:system}, for all $x_0\in\OO$ and all $J\in\s$ it holds that
\[\left\{
\begin{array}{l}
0=\langle f'_c(x_0)f'_s(x_0)^{c},1\rangle=\langle a_1b^c+(Ja_2)b^c,1\rangle\\
0=\langle f'_c(x_0)f'_s(x_0)^{c},J\rangle=\langle a_1b^c+(Ja_2)b^c,J\rangle
\end{array}
\right.\,,\]
where we have set $a_1:=\vs{(f'_c)}(x_0), a_2:=(f'_c)'_s(x_0), b:=f'_s(x_0)$. By Lemma~\ref{lem:scalarproduct}, the previous system is equivalent to
\[\left\{
\begin{array}{l}
\re(a_1b^c)-\langle \im(a_2b^c),J\rangle=0\\
\langle\im(a_1b^c),J\rangle +\re(a_2b^c)=0
\end{array}
\right.\,.
\]
Since $J\in\s$ is arbitrary, it follows that $a_1b^c=0=a_2b^c$. As a consequence of this reasoning, either $f'_s$ vanishes identically or $f'_c$ vanishes in a non-empty open subset of $\OO$, whence throughout $\OO$.\qedhere
\end{itemize}
\end{proof}

The next example shows that $N_f$ may well have dimension $7$.

\begin{example}\label{ex:square}
Consider the octonionic polynomial $f(x)=x^2$. Its slice derivative $f'_c(x)=2x$ vanishes exactly at $0$ and its spherical derivative $f'_s(x)=t(x)$ vanishes exactly in $\im(\oo)\setminus\{0\}$, while for $\alpha\neq0$
\[f'_c(\alpha+\beta J)f'_s(\alpha+\beta J)^{-1}=2(\alpha+\beta J)(2\alpha)^{-1}=1+\beta\alpha^{-1}J\]
never belongs to $\cc_J^\perp$. Thus,
\[N_f=D_f\cup V(f'_c)=\im(\oo)\,.\]
\end{example}

We are now in a position to prove that octonionic slice regular functions are quasi-open, according to theory presented in~\cite{titusyoung}.

\begin{definition}
A map $f:A\to B$ between topological spaces is \emph{quasi-open} if, for each $b\in B$ and for each open neighborhood $U$ of a compact connected component of $f^{-1}(b)$ in $A$, the point $b$ belongs to the interior of $f(U)$.
\end{definition}

\begin{theorem}\label{thm:quasi-open}
Let $\OO$ be either a slice domain or a product domain and let $f:\OO\to\oo$ be a slice regular function. If $f$ is not slice constant, then $f$ is quasi-open.
\end{theorem}

\begin{proof}
By Theorem~\ref{thm:jacobian}, $\det(df_{x_0})\geq0$ for all $x_0\in\OO$. Moreover, if $f$ is not slice constant then $N_f=\{x_0\in\OO \, | \, \det(df_{x_0})=0\}$ has dimension less than $8$ by the previous proposition. By~\cite[pages 91-92]{titusyoung}, $f$ is quasi-open.
\end{proof}


\section{Fibers and openness}\label{sec:fibers}

In this section, we are able to improve the Open Mapping Theorem for octonionic slice regular functions obtained in~\cite[Theorem 5.7]{gpsdivisionalgebras} (see also~\cite[Theorem 5.4]{wang}). We begin by studying the fibers of slice regular functions.

\begin{theorem}\label{thm:fibers}
Let $\OO$ be either a slice domain or a product domain. Let $f:\OO\to\oo$ be a nonconstant slice regular function and take $c\in f(\OO)$.
\begin{enumerate}
\item If $N(f-c)\not\equiv0$ then $f^{-1}(c)$ consists of isolated points or isolated $6$-spheres of the form $\s_{x_0}$. Moreover, the union of such $6$-spheres is $f^{-1}(c)\cap D_f$.
\item If $N(f-c)\equiv0$ then $f^{-1}(c)$ includes a real analytic subset of $\OO$, namely a $2$-surface $W_{f,c}$, such that 
\[f^{-1}(c)\setminus D_f=W_{f,c}\setminus D_f\,,\]
while $f^{-1}(c)\cap D_f$ is a (possibly empty) union of isolated $6$-spheres of the form $\s_{x_0}$. Clearly, $W_{f,c}\subseteq N_f$.
\end{enumerate}
Case {\it 2} is excluded if $\OO$ is a slice domain.
\end{theorem}

\begin{proof}
Let us set $g:=f-c$.

To prove property {\it 1}, we observe that if $N(g)\not\equiv0$ then $V(g)=f^{-1}(c)$ consists of isolated points or isolated $6$-spheres of the form $\s_x$ by Theorem~\ref{structureofzeros}. Moreover, when $x_0\in V(g)$, the inclusion $\s_{x_0}\subseteq V(g)$ holds if, and only if, $x_0\in D_g=D_f$.

Now let us prove property {\it 2}. If $N(g)\equiv0$ then, by Theorem~\ref{zerosonsphere}, each $6$-sphere $\s_{x_0}\subset \OO$ includes some zero of $g$. If $\s_{x_0}\subseteq V(g)$ then $\s_{x_0}\subseteq D_g=D_f$. Otherwise, by~\cite[Theorem 4.1]{gpsalgebra}, $g'_s(x_0)=f'_s(x_0)\neq0$ and the function $g$ has exactly one zero in $\s_{x_0}$; namely,
\[\re(x_0)-\vs g(x_0)g'_s(x_0)^{-1}=\re(x_0)+(c-\vs f(x_0))f'_s(x_0)^{-1}\,.\]
Thus, $V(g)\setminus D_g=f^{-1}(c)\setminus D_f$ is a real analytic $2$-surface $\Sigma$. Moreover, the circular closed subset $V(g)\cap D_g=f^{-1}(c)\cap D_f$ is a union of isolated $6$-spheres of the form $\s_{x_0}$ because its intersection with $\cc_I^+$ is discrete for some $I\in\s$ (see Theorem~\ref{structureofzeros}). We can now prove that $\Sigma$ extends analytically through each $6$-sphere $\s_{x_0}\subset V(g)\cap D_g=f^{-1}(c)\cap D_f$. Indeed, since $g\not\equiv0$, the spherical multiplicity of $g$ at $\s_{x_0}$ is a finite positive natural number $2n$ and
\[g(x)=\Delta_{x_0}^n(x)h(x)\]
for some $h\in\mc{SR}(\OO)$ that does not vanish identically in $\s_{x_0}$. Since $N(g)$ vanishes identically in $\OO$, so does $N(h)$. In particular, $h$ has a unique zero $w_0$ in $\s_{x_0}$ and $h'_s(x_0)\neq0$. Let $U$ be a circular neighborhood of $\s_{x_0}$ where $h'_s$ never vanishes. Then $h$ vanishes identically on the real analytic $2$-surface patch $\Upsilon$ formed by the points $\re(x)-\vs h(x)h'_s(x)^{-1}$ for $x\in U$. Now, $\Upsilon$ includes both the unique zero $w_0$ of $h$ in $\s_{x_0}$ and the $2$-surface $\Sigma\cap U$. Thus, we have extended $\Sigma$ analytically through $\s_{x_0}$.

Our final remark is the following. If $\OO$ is a slice domain then Proposition~\ref{SRnonsingular} tells us that $N(g)\equiv0$ only when $g\equiv0$. In this case $f\equiv c$, which is excluded by our hypothesis.
\end{proof}

\begin{definition}
Let $\OO$ be either a slice domain or a product domain and let $f:\OO\to\oo$ be a slice regular function. If, for some $c\in f(\OO)$, case {\it 2} of the previous theorem applies, we say that $f$ has a \emph{wing} $W_{f,c}$ and we denote the union of all wings of $f$ by $W_f$. Otherwise, we say that $f$ has no wings and we set $W_f:=\emptyset$.
\end{definition}

The function in Example~\ref{ex:square} has fibers of type {\it 1}, but no wings:

\begin{example}\label{ex:square2}
For $f(x)=x^2$: if $c$ belongs to the real half-line $(-\infty,0)$, then the fiber $f^{-1}(c)$ is the $6$-sphere $\sqrt{-c}\,\s$; the fiber $f^{-1}(0)$ over $0$ is the singleton $\{0\}$; all other fibers $f^{-1}(c)$ consist of two points.
\end{example}

If $f$ is slice constant, then every half-plane $\cc_J^+$ is a wing for $f$. This is the case for the function
\[\eta_{J_0}(x):=\frac12+\frac{\im(x)}{|\im(x)|}\frac{J_0}{2}\]
that appeared in Example~\ref{ex:1fin}. Less trivial examples of wings can be constructed by means of the next remark, as done over quaternions in~\cite[Examples 5.9]{gporientation}.

\begin{remark}\label{rmk:wingproducedbyeta}
Let $\OO$ be a product domain and let $g\in\mc{SR}(\OO)$. Fix $J_0\in\s$ and set $f:=g\cdot\eta_{J_0}$. By~\cite[Theorem 3.1]{gpsalgebra},
\[N(f)=N(g)N(\eta_{J_0})\equiv0\,.\]
Thus, $f$ has a wing $W_{f,0}$.
\end{remark}

\begin{example}\label{ex:1flatwing}
The function $f:\oo\setminus\rr\to\oo$ defined by
\[f(x):=2x\cdot\eta_{-i}(x)=x-x\frac{\im(x)}{|\im(x)|}i\]
has
\[N_f=W_{f,0}=\cc_{-i}^+\,.\]
Indeed, it has a wing $W_{f,0}=\cc_{-i}^+$. Moreover, by direct computation, the following equalities hold for all $\alpha,\beta\in\rr$ with $\beta>0$ and for all $J\in\s$:
\begin{align*}
&f(\alpha+\beta J)=\alpha+\beta i+J(\beta-\alpha i)\,,\\
&f'_s(\alpha+\beta J)=1-\frac{\alpha}{\beta}i\,,\\
&f'_c(\alpha+\beta J)=1-Ji\,,\\
&f'_c(\alpha+\beta J)f'_s(\alpha+\beta J)^c=1+\frac{\alpha}{\beta}(i+J)-Ji\,.
\end{align*}
By Corollary~\ref{cor:system}, $\det(df_{x_0})$ only vanishes when $x_0\in\cc_{-i}^+$.
\end{example}

The union $W_f$ of all wings may well have dimension $7$, as shown in the next example. Here and later in this work, we will use the notation $h^{\punto2}(x)=h(x)^{\punto2}:=h(x)\cdot h(x)$.

\begin{example}
The slice regular function on $\oo\setminus\rr$ defined by 
\[f(x):=x\cdot\eta_{i}(x)-x^{-1}\cdot\eta_{-i}(x)\]
has a wing $W_{f,c}$ for every $c$ in the unit $5$-sphere in $\cc_i^\perp=j\rr+k\rr+\ell\rr+\ell i\rr+\ell j\rr+\ell k\rr$. Indeed,
\begin{align*}
N(f)(x)&=(x\cdot\eta_{i}(x)-x^{-1}\cdot\eta_{-i}(x))\cdot(-x^{-1}\cdot\eta_{i}(x)+x\cdot\eta_{-i}(x))\\
&=-\eta_{i}^{\punto2}(x)-\eta_{-i}^{\punto2}(x)=-\eta_{i}(x)-\eta_{-i}(x)\\
&\equiv-1
\end{align*}
and
\begin{align*}
N(f-c)(x)&=N(f)(x)-f(x)\cdot c^c-c\cdot f^c(x)+n(c)\\
&=n(c)-1-2\langle\vs f(x),c\rangle-2\im(x)\langle f'_s(x),c\rangle\,.
\end{align*}
Thus, $N(f-c)\equiv0$ if, and only if, $\langle\vs f(x),c\rangle\equiv\frac{n(c)-1}{2}$ and $\langle f'_s(x),c\rangle\equiv0$. Noticing that $\vs f,f'_s$ are nonconstant and real analytic functions $\oo\setminus\rr\to\cc_i$, the last two equalities are equivalent to $c\in\cc_i^\perp,n(c)=1$.
\end{example}

As a consequence of Theorem~\ref{thm:fibers}, any slice regular function $f$ can be restricted to fulfill the next definition (see~\cite{titusyoung}).

\begin{definition}
A map $f:A\to B$ between topological spaces is \emph{light} if, for each $b\in B$, $f^{-1}(b)$ is totally disconnected.
\end{definition}

This allows us to prove the announced new version of the Open Mapping Theorem for octonionic slice regular functions. In the statement, $\overline{D}_f$ denotes the closure in $\OO$ of $D_f\subseteq\OO\setminus\rr$.

\begin{theorem}
Let $\OO$ be either a slice domain or a product domain and let $f:\OO\to\oo$ be a slice regular function. If $f$ is not slice constant, then its restriction to
\[\OO\setminus(\overline{D}_f\cup W_f)\]
is an open map. Moreover: if $W_f=\emptyset$, then the image $f(U)$ of any circular open subset $U$ of $\OO$ is open; in particular, $f(\OO)$ is open.
\end{theorem}

\begin{proof}
Let $\OO':=\OO\setminus(\overline{D}_f\cup W_f)$ and observe that $f_{|_{\OO'}}$ is a light function as a consequence of Theorem~\ref{thm:fibers}. Moreover, $f_{|_{\OO'}}$ is of class $C^\omega$ and $\det(df_{x_0})\geq0$ for all $x_0\in\OO'$ as proven in Theorem~\ref{thm:jacobian}. By~\cite[Theorem 2]{titusyoung}, it follows that $f_{|_{\OO'}}$ is an open map.

Now let us take the additional assumption that $W_f=\emptyset$. Pick any circular open subset $U$ of $\OO$ and any $b\in f(U)$. At least one connected component $C$ of $f^{-1}(b)$ intersects $U$. Moreover, by Theorem~\ref{thm:fibers}, $C$ is a compact set entirely contained in the circular open set $U$. Now, Theorem~\ref{thm:quasi-open} implies that $b$ is an interior point of $f(U)$.
\end{proof}

We point out that closing $D_f$ means adding to it a discrete set at most.

\begin{remark}
Under the hypotheses of the previous theorem, $N_f$ is a closed subset of $\OO$ and $V(f'_s)$ is a closed subset of $\OO\setminus\rr$, included in $N_f$. As a consequence,
\[\overline{D}_f\setminus D_f\subseteq N_f\cap\rr=V(f'_c)\cap\rr\,.\]
If $f$ is not slice constant, then $\overline{D}_f\setminus D_f$ is a closed and discrete subset of $\OO\cap\rr$.
\end{remark}

The present version of the Open Mapping Theorem is sharp. In the next example, $f$ is not an open mapping unless $\overline{D}_f$ is removed from $\OO$. The same was true over quaternions, see~\cite[page 814]{open}.

\begin{example}
The function $f(x)=x^2$ of Examples~\ref{ex:square} and~\ref{ex:square2} has $\overline{D}_f=\im(\oo)$. The ball $B(i,1)$ includes $i\in\overline{D}_f$, but it does not intersect $\cc_J$ for any $J\in\s,J\perp i$. Thus, while $f(B(i,1))$ includes $-1$, it does not include any point of $\cc_J\setminus\rr$ for any $J\in\s,J\perp i$. As a consequence, $f(B(i,1))$ is not an open subset of $\oo$.
\end{example}

In the following Example, $f$ is not an open mapping, unless $W_f$ is removed from $\OO$, and $f(\OO)$ is not an open subset of $\oo$.

\begin{example}
The map $f(x):=2x\cdot\eta_{-i}(x)$ of Example~\ref{ex:1flatwing} has $\OO=\oo\setminus\rr$ and $N_f=W_f=W_{f,0}=\cc_{-i}^+$. It fulfills the following equalities for all $\alpha,\beta\in\rr$ with $\beta>0$, for all $J\in\s$ and for $J':=J-\langle J,i\rangle i$:
\[f(\alpha+\beta J)=\alpha+\beta i+\beta J-\alpha Ji=\alpha(1+\langle J,i\rangle)+\beta(1+\langle J,i\rangle)i+\beta J'-\alpha J'i\,.\]
The open set $f(\OO\setminus W_f)$ is included in the half-space $\{x\in\oo\, | \,\langle x,i\rangle>0\}$, while $f(W_f)=\{0\}$. Thus, $f(\OO)=f(\OO\setminus W_f)\cup\{0\}$ is not an open set.
\end{example}

For the sake of completeness, we include in this section a study of the wings of a slice regular function.

\begin{remark}
We saw in Theorem~\ref{thm:fibers} that a wing $W_{f,c}$ can only exist if $f$ is a slice regular function on a product domain, i.e., an $\OO_D$, where $D$ is an open subset of $\cc$ that does not intersect $\rr$ and has two connected components $D^+,D^-$, switched by complex conjugation. By direct inspection in the proof, there exists an injective real analytic map $\omega:D^+\to\OO_D$ such that $\omega(D^+)=W_{f,c}$. The map $\omega$ can be constructed as follows: fix any $I\in\s$ and set, for each $\alpha+i\beta\in D^+$,
\begin{equation}\label{eq:omega1}
\omega(\alpha+i\beta):=\alpha+(c-\vs f(\alpha+\beta I))f'_s(\alpha+\beta I)^{-1}
\end{equation}
if $f'_s(\alpha+\beta I)\neq0$ and
\begin{equation}\label{eq:omega2}
\omega(\alpha+i\beta):=\alpha-\vs h(\alpha+\beta I)h'_s(\alpha+\beta I)^{-1}
\end{equation}
if, instead, $f(x)=c+\Delta_{\alpha+\beta I}^n(x)h(x)$ for some $n>0$ and some $h\in\mc{SR}(\OO_D)$ that does not vanish identically in $\alpha+\beta\s$. We note that, by construction, the map $\omega$ is independent of the choice of $I\in\s$.
\end{remark}

Wings are studied in further detail in the next theorem.  

\begin{theorem}
Let $f:\OO_D\to\oo$ be a slice regular function admitting a wing $W_{f,c}$ and let $\omega:D^+\to\OO_D$ be the map described in the previous remark.
\begin{enumerate}
\item $\omega$ is a real analytic embedding and the wing $W_{f,c}$ intersects transversally each sphere $\alpha_0+\beta_0\s$ in $\OO_D$ at the point $\omega(\alpha_0+i\beta_0)$.
\item If $x_0=\alpha_0+\beta_0 J\in W_{f,c}$ then $d\omega_{\alpha_0+i\beta_0}$ maps the vectors $1,i$ to the vectors 
\[1-ab^{-1}, J-(Ja)b^{-1} \in T_{x_0}W_{f,c}\,,\]
where $a:=f'_c(x_0),b:=f'_s(x_0)$ if $x_0\not\in D_f$ and $a:=h'_c(x_0),b:=h'_s(x_0)$ if instead $f(x)=c+\Delta_{x_0}^n(x)h(x)$ for some $n>0$ and some $h\in\mc{SR}(\OO_D)$ that does not vanish identically in $\alpha_0+\beta_0\s$. In particular, $a\neq b$.
\item If, for each $x_0=\alpha_0+\beta_0 J\in W_{f,c}$, we define $\mc{J}_{x_0}$ to be the restriction of $R_{b^{-1}} \circ L_{J} \circ R_{b}$ to $T_{x_0}W_{f,c}$, then $\omega$ is a biholomorphism between the Riemann surfaces $(D^+,i)$ and $(W_{f,c},\mc{J})$.
\end{enumerate}
\end{theorem}

\begin{proof}
We prove one fact at a time.
\begin{enumerate}
\item We may construct a commutative diagram
\[\begin{tikzcd}
D^+\arrow{dr}{\omega} \arrow{r}{\Phi}
&D^+\times\s \arrow{d}{\Psi} \\
&\OO_D
\end{tikzcd}\]
by setting $\Phi(z)=(z,\phi(z))$ with $\phi(\alpha+i\beta):=(\omega(\alpha+i\beta)-\alpha)\beta^{-1}$ and $\Psi(\alpha+i\beta,J):=\alpha+\beta J$. Clearly, $\Phi$ is a real analytic embedding and $\Psi$ is a real analytic isomorphism with inverse $\Psi^{-1}(x)=(\re(x)+i|\im(x)|,\im(x)|\im(x)|^{-1})$. Moreover, $\Psi$ maps each product $\{\alpha+i\beta\}\times\s$ onto the $6$-sphere $\alpha+\beta\s$. It follows that $\omega$ is an embedding and $W_{f,c}$ is transverse in $\OO_D$ to each sphere $\alpha+\beta\s$ at $\omega(\alpha+i\beta)$.
\item We first consider $x_0=\omega(\alpha_0+i\beta_0)\in W_{f,c}\setminus D_f$. Fix $I\in\s$. From~\eqref{eq:omega1}, we derive the equality
\[\omega(\alpha+i\beta)f'_s(\alpha+\beta I)=\alpha f'_s(\alpha+\beta I)+c-\vs f(\alpha+\beta I)\,,\]
whence
\begin{align*}
&\frac{\partial \omega}{\partial \alpha}(\alpha+i\beta) f'_s(\alpha+\beta I) =f'_s(\alpha+\beta I) + (\alpha-\omega(\alpha+i\beta)) \frac{\partial f'_s(\alpha+\beta I)}{\partial \alpha}-\frac{\partial \vs f(\alpha+\beta I)}{\partial \alpha}\,,\\\nonumber
&\frac{\partial \omega}{\partial \beta}(\alpha+i\beta) f'_s(\alpha+\beta I) = (\alpha-\omega(\alpha+i\beta)) \frac{\partial f'_s(\alpha+\beta I)}{\partial \beta}-\frac{\partial \vs f(\alpha+\beta I)}{\partial \beta}\,.
\end{align*}
Using Formulas~\eqref{eq:partialderivatives}, we conclude that
\begin{align*}
\frac{\partial \omega}{\partial \alpha}(\alpha+i\beta) f'_s(\alpha+\beta I) &= f'_s(\alpha+\beta I) - \vs{(f'_c)}(\alpha+\beta I) + (\alpha-\omega(\alpha+i\beta)) (f'_c)'_s(\alpha+\beta I)\\
&=f'_s(\alpha+\beta I) - f'_c(\omega(\alpha+i\beta))\,,\\
\frac{\partial \omega}{\partial \beta}(\alpha+i\beta) f'_s(\alpha+\beta I) &= \frac{\omega(\alpha+i\beta)-\alpha}{\beta}f'_s(\alpha+\beta I) +  \frac{\alpha-\omega(\alpha+i\beta)}{\beta} \vs{(f'_c)}(\alpha+\beta I)\\
&\phantom{=}+ \beta(f'_c)'_s(\alpha+\beta I)\\
&=\frac{\omega(\alpha+i\beta)-\alpha}{\beta} \left(f'_s(\alpha+\beta I) - f'_c(\omega(\alpha+i\beta))\right)\,.
\end{align*}
At $\alpha_0+i\beta_0$, using the equalities $\omega(\alpha_0+i\beta_0)=x_0$ and $f'_s(\alpha_0+\beta_0I)=f'_s(x_0)$, we conclude that
\begin{align*}
&\frac{\partial \omega}{\partial \alpha}(\alpha_0+i\beta_0) = 1 -f'_c(x_0)f'_s(x_0)^{-1}=1-ab^{-1}\,,\\
&\frac{\partial \omega}{\partial \beta}(\alpha_0+i\beta_0) = J -(Jf'_c(x_0))f'_s(x_0)^{-1}=J-(Ja)b^{-1}\,.
\end{align*}
with $a:=f'_c(x_0), b:=f'_s(x_0)$.

Now we consider $x_0=\omega(\alpha_0+i\beta_0)\in W_{f,c}\cap D_f$, so that $f(x)=c+\Delta_{x_0}^n(x)h(x)$ for some $n>0$ and some $h\in\mc{SR}(\OO_D)$ that does not vanish identically in $\alpha+\beta\s$. From~\eqref{eq:omega2}, we derive the equality
\[\omega(\alpha+i\beta)h'_s(\alpha+\beta I)=\alpha h'_s(\alpha+\beta I)-\vs h(\alpha+\beta I)\,.\]
Reasoning as before, we can prove that $\frac{\partial \omega}{\partial \alpha}(\alpha_0+i\beta_0)=1-ab^{-1}$ and $\frac{\partial \omega}{\partial \beta}(\alpha_0+i\beta_0)=J-(Ja)b^{-1}$ with $a:=h'_c(x_0), b:=h'_s(x_0)$.
\item Since $\omega:D^+\to W_{f,c}$ a real analytic isomorphism, we can consider the complex structure $\mc{J}$ defined on $W_{f,c}$ as the push-forward of the standard complex structure on $D^+$ via $\omega$. By construction, $\omega:(D^+,i)\to(W_{f,c},\mc{J})$ is a biholomorphism. Taking into account our previous computations, at each $x_0=\alpha_0+\beta_0 J\in W_{f,c}$ the structure $\mc{J}_{x_0}$ maps $\frac{\partial \omega}{\partial \alpha}(\alpha_0+i\beta_0)=1-ab^{-1}$ to $\frac{\partial \omega}{\partial \beta}(\alpha_0+i\beta_0)=J-(Ja)b^{-1}$ and the latter to $-\frac{\partial \omega}{\partial \alpha}(\alpha_0+i\beta_0)=-1+ab^{-1}$. In other words, $\mc{J}_{x_0}$ coincides with the restriction of $R_{b^{-1}} \circ L_{J} \circ R_{b}$ to $T_{x_0}W_{f,c}$.
\end{enumerate}
\end{proof}

\begin{example}\label{ex:1wing}
The slice regular function on $\oo\setminus\rr$ defined by 
\[f(x):=2(x^2+xj+\ell)\cdot\eta_{-i}(x)\]
has a wing $W_{f,0}$ by Remark~\ref{rmk:wingproducedbyeta}. If $g(x)=x^2+xj+\ell$ then, for all $\alpha,\beta\in\rr$ with $\beta>0$ it holds that
\begin{align*}
&\vs g(\alpha+\beta J)=\alpha^2-\beta^2+\alpha j+\ell,\\
&g'_s(\alpha+\beta J)=2\alpha+j,\\
&\vs{(2\eta_{-i})}(\alpha+\beta J)\equiv1,\\
&(2\eta_{-i})'_s(\alpha+\beta J)=-\beta^{-1}i,\\
&\vs f(\alpha+\beta J)=\alpha^2-\beta^2+2\alpha\beta i+\alpha j-\beta k+\ell,\\
&f'_s(\alpha+\beta J)=2\alpha-(\alpha^2-\beta^2)\beta^{-1}i+j+\alpha\beta^{-1}k-\beta^{-1}\ell i\,,
\end{align*}
where the last two equalities take into account Formulas~\eqref{eq:leibnizspherical}. The parametrization $\omega:\cc^+\to W_{f,0}$ can be computed explicitly as
\[\omega(\alpha+i\beta)=\alpha-\frac{\beta}{1+r+r^2}\left((-1-r+r^2)i+2\beta rj-2\alpha r k+4\alpha\beta\ell+2(\alpha^2-\beta^2)\ell i\right)\,,\]
where $r:=\alpha^2+\beta^2$. We notice that the linear span of the elements of $W_{f,0}$ is the $6$-dimensional vector space $\rr+i\rr+j\rr+k\rr+\ell\rr+\ell i\rr$, whence there is no subalgebra of $\oo$ isomorphic to $\hh$ that includes $W_{f,0}$.
\end{example}


\section{Octonionic slice regular functions and branching}\label{sec:branchedcoverings}

In this section, we prove that the branch set of a slice regular function is its singular set. This allows us to complete the results of Section~\ref{sec:inducedacs} about induced almost-complex structures. We begin with three preliminary results. The first one is the octonionic analog of~\cite[Remark 3.5]{ocs}.

\begin{lemma}\label{lem:stereographicprojection}
Fix $x_0=\alpha_0+\beta_0J$ with $\alpha_0,\beta_0 \in \rr, \beta_0>0$ and $J \in \s$. Consider the map
\[\Theta:\quad\s_{x_0}\setminus\{x_0^c\}\to\cc_J^\perp\,,\quad x\mapsto(x-x_0)(x-x_0^c)^{-1}\]
and the affine transformation $\rho:\ \cc_J^\perp\to \alpha_0+\cc_J^\perp\,,\ x\mapsto \alpha_0-\beta_0Jx$. Then $\Theta$ is an orientation-preserving conformal transformation and $\rho\circ\Theta$ is the stereographic projection of $\s_{x_0}$ from the point $x_0^c$ to $\alpha_0+\cc_J^\perp$, which is the affine plane tangent to $\s_{x_0}$ at $x_0$. 
\end{lemma}

\begin{proof}
For $J\in\s$ fixed, consider the map from $\s\setminus\{-J\}$ to $\cc_J^\perp$ defined as
\[L\mapsto(L-J)(L+J)^{-1}=\frac{J\times L}{1+\langle J,L\rangle}\,.\]
When composed with the rotation $x\mapsto-Jx$ of $\cc_J^\perp$, it is the stereographic projection of $\s$ from the point $-J$ to the tangent plane $\cc_J^\perp$ at $J$. Indeed, every $L\in \s\setminus\{-J\}$ can be expressed as $L=\sin(\theta) I+\cos(\theta) J$ for some $I\in\s$ orthogonal to $J$ and some $\theta\in[0,\pi)$. Now, for $x=\frac{J\times L}{1+\langle J,L\rangle}=\frac{\sin(\theta)}{1+\cos(\theta)}JI$, the points $-J,L,-Jx$ are aligned: $-Jx+J=\frac{\sin(\theta) I+ (1+\cos(\theta))J}{1+\cos(\theta)}$ is a real rescaling of $L+J=\sin(\theta) I+ (1+\cos(\theta))J$.

Our thesis follows by applying a real dilation and a real translation to transform $\s$ into $\s_{x_0}$.
\end{proof}

The second preliminary result is a nonassociative generalization of~\cite[Proposition 3.6]{ocs}.

\begin{theorem}\label{thm:multiplicityatsingularity}
Let $f:\OO\to\oo$ be a slice regular function and fix $x_0\in\OO$. The point $x_0$ belongs to the singular set $N_f$ if, and only if, there exists $\widetilde x_0\in\s_{x_0}$ such that
\[f(x)=f(x_0)+(x-x_0)\cdot((x-\widetilde x_0)\cdot h(x))\]
for some slice regular function $h$ on $\OO$.
\end{theorem}

\begin{proof}
By~\cite[Theorem 7.2]{gpssingularities} (which generalized~\cite{expansionsalgebras}), an expansion
\begin{align*}
f(x)&=A_0+(x-x_0)\cdot A_1+\Delta_{x_0}(x)\cdot A_2+\Delta_{x_0}(x)\cdot(x-x_0)\cdot A_3+\ldots
\end{align*}
with $A_0,A_1,A_2,A_3,\ldots\in\oo$ is possible. The series on the right-hand side of the previous equality converges absolutely and uniformly on compact sets in an open neighborhood of $\s_{x_0}$ in $\OO$, where its sum equals $f(x)$. Thus, there exists a slice regular function $f_1:\OO\to\oo$ such that $f(x)=A_0+(x-x_0)\cdot f_1(x)$ in $\Omega$. By Theorem~\ref{thm:factorization}, $f(x_0)=A_0$. Moreover,  there exists a slice regular function $f_2:\OO\to\oo$ such that $f_1(x)=A_1+(x-x_0^c)\cdot f_2(x)$ in $\Omega$ and it holds that $f_1(x_0^c)=A_1$ (which yields $f'_s(x_0)=A_1$ in case $x_0\not\in\rr$). Finally, there exists a slice regular function $f_3:\OO\to\oo$ such that $f_2(x)=A_2+(x-x_0)\cdot f_3(x)$ in $\Omega$ and it holds that $f_2(x_0)=A_2$. Overall,
\[f(x)=f(x_0)+(x-x_0)\cdot A_1+\Delta_{x_0}(x)\cdot A_2+\Delta_{x_0}(x)\cdot(x-x_0)\cdot f_3(x)\]
in $\OO$. As a byproduct, we get that
\[f(x)=f(x_0)+(x-x_0)\cdot (A_1+2\im(x_0)A_2)+(x-x_0)^{\punto2}\cdot (A_2+(x-x_0^c)\cdot f_3(x))\]
in $\OO$. Using the Leibniz rule~\eqref{eq:leibnizcullen} and Theorem~\ref{thm:factorization}, we conclude that $f'_c(x_0)=A_1+2\im(x_0)A_2$.

It holds that $x_0\in D_f$ (or $x_0\in N_f\cap\rr$) if, and only if, $A_1=0$, i.e.,
\begin{align*}
f(x)&=f(x_0)+\Delta_{x_0}(x)\cdot A_2+\Delta_{x_0}(x)\cdot(x-x_0)\cdot f_3(x)\\
&=f(x_0)+(x-x_0)\cdot((x-x_0^c)\cdot(A_2+(x-x_0)\cdot f_3(x)))\,.
\end{align*}

We now characterize the case $x_0\in N_f\setminus D_f$. For $x_0\in\cc_J$, that condition is equivalent to $f'_c(x_0)f'_s(x_0)^{-1}\in\cc_J^\perp$ by Theorem~\ref{thm:rankofdifferential}. This is, in turn, equivalent to $1+(2\im(x_0)A_2)A_1^{-1}\in\cc_J^\perp$ or, taking into account Lemma~\ref{lem:scalarproduct}, $1+2\im(x_0)(A_2A_1^{-1})\in\cc_J^\perp$. By Lemma~\ref{lem:stereographicprojection}, the latter happens if, and only if, there exists $\widetilde x_0\in\s_{x_0}\setminus\{x_0^c\}$ such that
\[1+2\im(x_0)(A_2A_1^{-1})=(\widetilde x_0-x_0)(\widetilde x_0-x_0^c)^{-1}\,,\]
i.e., $A_2A_1^{-1}=(x_0^c-\widetilde x_0)^{-1}$. The latter is equivalent to
\begin{align*}
f(x)&=f(x_0)+(x-x_0)\cdot ((x-\widetilde x_0)\cdot A_2+\Delta_{x_0}(x)\cdot f_3(x))\\
&=f(x_0)+(x-x_0)\cdot ((x-\widetilde x_0)\cdot (A_2+(x-\widetilde x_0^c)\cdot f_3(x)))\,.
\end{align*}
This concludes the proof.
\end{proof}

\begin{proof}[Alternate proof of Theorem~\ref{thm:multiplicityatsingularity}]
Since $\widetilde f(x):=f(x)-f(x_0)$ vanishes at $x=x_0$, Theorem~\ref{thm:factorization} guarantees that there exists a slice regular function $g:\OO\to\oo$ such that
\[\widetilde f(x)=(x-x_0)\cdot g(x).\]
For all $x\in\OO$, the Leibniz rule~\eqref{eq:leibnizcullen} yields
\[f'_c(x)=\widetilde f'_c(x)=g(x)+(x-x_0)\cdot g'_c(x),\quad f'_c(x_0)=g(x_0)\,,\]
where the last equality is true by Theorem~\ref{thm:factorization}. For all $x\in\OO\setminus\rr$, the Leibniz rule~\eqref{eq:leibnizspherical} yields
\[f'_s(x)=\widetilde f'_s(x)=g^\circ_s(x)+(\re(x)-x_0)g'_s(x),\quad f'_s(x_0)=g^\circ_s(x_0)-\im(x_0)g'_s(x_0)\,.\] 

If $x\in\OO\cap\rr$, then the real differential $df_{x_0}$ is singular if, and only if, $0=f'_c(x_0)=g(x_0)$. This is, in turn, equivalent to the existence of a slice regular $h:\OO\to\oo$ such that
\[g(x)=(x-x_0)\cdot h(x)\,,\]
i.e.,
\[\widetilde f(x)=(x-x_0)\cdot((x-x_0)\cdot h(x))\,,\]
as desired.

Suppose, instead, $x_0\in\OO\setminus\rr$. By Corollary~\ref{cor:system}, the real differential $df_{x_0}$ is singular if, and only if, $\langle f'_c(x_0)f'_s(x_0)^{c},1\rangle=0$ and $\langle f'_c(x_0)f'_s(x_0)^{c},\im(x_0)\rangle=0$. Let us express these two scalar products in a different form. We compute:
\begin{align*}
&f'_c(x_0)f'_s(x_0)^c=g(x_0)(g^\circ_s(x_0)-\im(x_0)g'_s(x_0))^c\\
&=\big(g^\circ_s(x_0)+\im(x_0)g'_s(x_0)\big)\big(g^\circ_s(x_0)^c+g'_s(x_0)^c\im(x_0)\big)\\
&=n(g^\circ_s(x_0))+\im(x_0)^2n(g'_s(x_0))+2\im\big((\im(x_0)g'_s(x_0))g^\circ_s(x_0)^c\big)\,.
\end{align*}
Now,
\[\langle f'_c(x_0)f'_s(x_0)^{c},1\rangle=n(g^\circ_s(x_0))+\im(x_0)^2n(g'_s(x_0))=(N(g))^\circ_s(x_0)\,,\]
where we have applied Formulas~\eqref{eq:normalspherical} to $g$. Moreover,
\[\langle f'_c(x_0)f'_s(x_0)^{c},\im(x_0)\rangle=2\langle(\im(x_0)g'_s(x_0))g^\circ_s(x_0)^c,\im(x_0)\rangle=2|\im(x_0)|^2\langle g'_s(x_0)g^\circ_s(x_0)^c,1\rangle\]
thanks to Lemma~\ref{lem:scalarproduct}. By applying Formulas~\eqref{eq:normalspherical} to $g$, we conclude that
\[\langle f'_c(x_0)f'_s(x_0)^{c},\im(x_0)\rangle=|\im(x_0)|^2t(g'_s(x_0)g^\circ_s(x_0)^c)=|\im(x_0)|^2(N(g))'_s(x_0)\,.\]
Thus, $df_{x_0}$ is singular if, and only if, $(N(g))^\circ_s(x_0)=0=(N(g))'_s(x_0)$, which is equivalent to $N(g)_{|_{\s_{x_0}}}\equiv0$. The last equality holds true if, and only if, there exist $\widetilde x_0\in\s_{x_0}$ and a slice regular function $h:\OO\to\oo$ such that
\[g(x)=(x-\widetilde x_0)\cdot h(x)\,.\]
This equality is, in turn, equivalent to
\[\widetilde f(x)=(x-x_0)\cdot((x-\widetilde x_0)\cdot h(x))\,,\]
which is our thesis.
\end{proof}

In other words, a point $x_0$ belongs to the singular set $N_f$ if, and only if, the total multiplicity of $f-f(x_0)$ at $\s_{x_0}$ is greater than $1$.

\begin{lemma}
Let $\OO$ be either a slice domain or a product domain and let $f:\OO\to\oo$ be a slice regular function. If $x_0\in N_f\setminus W_f$, then the total multiplicity of $f-f(x_0)$ at $\s_{x_0}$ is a finite number $n\geq2$. If, moreover, $x_0\not\in D_f$ and if $U_0$ is any neighborhood of $x_0$ in $\OO$, then there exist neighborhoods $U_1,U_2$ of $x_0$ with $U_0\supseteq U_1\supseteq U_2$ and with the following property: for all $x_1\in U_2$, the function $f-f(x_1)$ has finitely many zeros in $U_1$, with total multiplicities whose sum equals $n$. 
\end{lemma}

\begin{proof}
Since $x_0\not\in W_f$, Theorem~\ref{thm:fibers} guarantees that $N(f-f(x_0))\not\equiv0$, whence the total multiplicity of $f-f(x_0)$ at $\s_{x_0}$ is a finite number $n$. On the other hand, since $x_0\in N_f$, Theorem~\ref{thm:multiplicityatsingularity} implies that $n\geq2$.

Now suppose $x_0\not\in D_f$ and let $U_0$ be any neighborhood of $x_0=\alpha_0+\beta_0J_0$ in $\OO$. For all $r>0$, consider the closed disc $D^r:=\overline{\Delta(\alpha_0+i\beta_0,r)}\subset\cc$ and its circularization $T^r=\OO_{D^r}$, which is a neighborhood of $\s_{x_0}$. Let us also consider the cone $C^r=\bigcup_{|J-J_0|<r}\cc_J$. We use different arguments depending on whether or not $x_0$ belongs to the real axis.
\begin{itemize}
\item Suppose $x_0\in\OO\setminus\rr$. There exists $r_0$ such that: the inclusions $T^{r_0}\subset\OO\setminus\rr, T^{r_0}\cap C^{2r_0}\subseteq U_0$ hold; $N(f-f(x_0))$ never vanishes in $T^{r_0}\setminus\s_{x_0}$; and $f'_s$ never vanishes in $T^{r_0}$. Let
\[m:=\min_{x\in T^{r_0}}|\im(x)f'_s(x)|\,.\]
There exists $r_1$ with $0<r_1\leq r_0$ such that
\[|\vs f(\alpha+\beta J)-\vs f(\alpha'+\beta'J')|+|\beta f'_s(\alpha+\beta J)-\beta'f'_s(\alpha'+\beta'J')|\leq r_0 m\]
for all $\alpha+\beta J,\alpha'+\beta'J'\in T^{r_1}$ with $\beta,\beta'\geq0$.

We claim that there exists $r_2$ with $0<r_2\leq r_1$ having the following property: for each $x_1\in T^{r_2}\cap C^{r_2}$, the distinct zeros $x_1,\ldots x_h$ of the function $f-f(x_1)$ in $T^{r_1}$ have multiplicities whose sum equals $n$.

Our claim can be proven as follows. Let us denote the restriction of $N(f-f(x_1))$ to $(T^{r_1})^+_J$ by $\phi_{x_1}$ and think of it as a holomorphic function of one complex variable. Since $\phi_{x_0}$ has multiplicity $n$ at $x_0$ and no other zeros, there exists $r_2$ with $0<r_2\leq r_1$ such that the sum of the multiplicities of the zeros of $\phi_{x_1}$ equals $n$ for all $x_1\in T^{r_2}\cap C^{r_2}$. If this were not true, we could construct a sequence of holomorphic functions contradicting Hurwitz's Theorem~\cite[Theorem 2.5]{libroconway}. Thus, the sum of the total multiplicities of the zeros of $f-f(x_1)$ in $T^{r_1}$ equals $n$ for any $x_1\in T^{r_2}\cap C^{r_2}$.

The claim thus established, for each $k\in\{1,\ldots,h\}$, let $\alpha_k,\beta_k\in\rr$ (with $\beta_k\geq0$) be such that $x_k=\alpha_k+J_k\beta_k$. Each equality $f(x_k)=f(x_1)$ implies that
\[J_k=\left(\vs f(x_1)-\vs f(x_k)+\beta_1J_1f'_s(x_1)\right)(\beta_kf'_s(x_k))^{-1}\,,\]
whence
\[J_k-J_1=\left(\vs f(x_1)-\vs f(x_k)+J_1(\beta_1f'_s(x_1)-\beta_kf'_s(x_k))\right)(\beta_kf'_s(x_k))^{-1}\]
and $|J_k-J_1|<r_0 m m^{-1}=r_0$. It follows that $|J_k-J_0|\leq|J_k-J_1|+|J_1-J_0|<r_0+r_2$, whence $x_1,\ldots,x_k$ all belong to $T^{r_1}\cap C^{r_0+r_2}$. Now consider the neighborhoods $U_1:=T^{r_1}\cap C^{r_0+r_2}$ and $U_2:=T^{r_2}\cap C^{r_2}$ of $x_0$. It holds that $U_0\supseteq T^{r_0}\cap C^{2r_0}\supseteq U_1\supseteq U_2$, as desired.
\item If $x_0\in\OO\cap\rr$, then $\s_{x_0}=\{x_0\}$ and each $T^r$ is the Euclidean ball of radius $r$ centered at $x_0$. There exists $r_0$ such that $T^{r_0}$ is included in $U_0$ and such that $N(f-f(x_0))$ never vanishes in $T^{r_0}\setminus\{x_0\}$. Arguing as above, we can prove that there exist $r_1,r_2$ with $0<r_2\leq r_1\leq r_0$ with the following property: for all $x_1\in T^{r_2}$, the distinct zeros $x_1,\ldots x_h$ of the function $f-f(x_1)$ in $T^{r_1}$ have multiplicities whose sum equals $n$. If we set $U_1:=T^{r_1}$ and $U_2:=T^{r_2}$, the thesis immediately follows.\qedhere
\end{itemize}
\end{proof}

We are now in a position to prove the result we announced at the beginning of this section.

\begin{theorem}
Let $\OO$ be either a slice domain or a product domain and let $f:\OO\to\oo$ be a slice regular function. Then $N_f$ is the branch set of $f$. More precisely:
\begin{enumerate}
\item for every point $x_0\in\OO\setminus N_f$, there exists an open neighborhood $U$ of $x_0$ in $\OO$ such that $f(U)$ is open and $f_{|_U}:U\to f(U)$ is a diffeomorphism; and
\item for every point $x_0\in N_f$ and for every neighborhood $U$ of $x_0$ in $\OO$, $f_{|_U}$ is not injective.
\end{enumerate}
If, moreover, the restriction $f_{|_{\Omega\setminus N_f}}:\OO\setminus N_f \to f(\OO\setminus N_f)$ is proper, then $f:\OO\to f(\OO)$ is a branched covering.
\end{theorem}

\begin{proof}
If $f$ is slice constant, then property {\it 2} holds at each point of $N_f=\Omega$. We assume henceforth $f$ not to be slice constant. Then its singular set $N_f$ is a closed subset of $\OO$ whose interior is empty by Proposition~\ref{prop:singularset}. If $x_0\in\OO\setminus N_f$, then the Implicit Function Theorem implies that there exists an open neighborhood $U$ of $x_0$ in $\OO$ such that $f_{|_U}:U\to f(U)$ is a diffeomorphism. Let us prove that every $x_0\in N_f$ is a branch point, i.e., property {\it 2}.

If $x_0$ belongs to $D_f\subseteq N_f$ then it is a branch point because $f$ is constant on the $6$-sphere $\s_{x_0}$ through $x_0$.

If $x_0$ belongs to $W_f\subseteq N_f$ then it is a branch point because $f$ is constant on the wing $W_{f,f(x_0)}$, which is a $2$-surface through $x_0$.

If $x_0\in N_f\setminus(D_f\cup W_f)$, we can prove that $x_0$ is a branch point as follows. If $U_0$ is any neighborhood of $x_0$ in $\OO$ then, by the previous lemma, there exist a number $n\geq2$ and neighborhoods $U_1,U_2$ of $x_0$ (with $U_0\supseteq U_1\supseteq U_2$) such that, for all $x_1\in U_2$, the sum of the total multiplicities of the zeros of $f-f(x_1)$ in $U_1$ equals $n$. We can observe that, for all $x_1\in U_2\setminus N_f$, the total multiplicity of $f-f(x_1)$ at $x_1$ equals $1$ by Theorem~\ref{thm:multiplicityatsingularity}. Thus, $f-f(x_1)$ vanishes not only at $x_1$ but also at some other point of $U_1$. In particular, $f$ is not injective in $U_0$.

Since the fibers of the restriction $f_{|_{\Omega\setminus N_f}}$ are discrete, if this restriction is a proper map from $\Omega\setminus N_f$ to $f(\Omega\setminus N_f)$, then it is a covering. In such a case, $f:\Omega\to f(\Omega)$ is a branched covering with branch set $N_f$.
\end{proof}

We can draw the following consequence about the push-forwards of the almost-complex structure $\J$ presented in Definition~\ref{def:standardstructure}.

\begin{corollary}
Let $g:\OO\to\oo$ be a slice regular function and pick an open subset $U$ of $\OO\setminus\rr$. The push-forward
\[\J^f_{f(x_0)}:=df_{x_0}\circ\J_{x_0}\circ df_{x_0}^{-1}\]
of the structure $\J_{|_U}$ through $f:=g_{|_U}$ is well-defined on $f(U)$ if, and only if, $f$ is injective.
\end{corollary}

This allows us to complete the results of Section~\ref{sec:inducedacs} about induced almost-complex structures. Indeed, Theorem~\ref{thm:inducedacs}, Theorem~\ref{thm:inducedoacs} and Remark~\ref{rmk:oneslicepreserving} have the following consequence.

\begin{corollary}
Let $g:\OO\to\oo$ be a slice regular function that is not slice constant and pick an open subset $U$ of $\OO\setminus\rr$ such that $f:=g_{|_U}$ is injective. If $x_0\in U_J$, then the equality
\[\J^f_{f(x_0)}(p+q):=Jp + (J(qf'_s(x_0)^{-1}))f'_s(x_0)\]
holds for all $p\in\cc_J\,f'_c(x_0)$ and all $q\in\cc_J^\perp\,f'_s(x_0)$. The push-forward $\J^f$ is an almost-complex structure on $f(U)$ and $f$ is a holomorphic map between the almost-complex manifolds $(U,\J)$ and $(f(U),\J^f)$. Moreover, $\J^f$ is orthogonal if, and only if, $(x_0,f'_c(x_0),f'_s(x_0))=0$ for all $x_0\in U$ (which is always the case if $g(\Omega_J)\subseteq\cc_J$ for some $J\in\s$). Finally, $\J^f=\J$ if $g$ is slice preserving (and the converse implication holds when $U=\OO\setminus\rr$).
\end{corollary}

We point out that, since $\J$ is not integrable, the induced structure $\J^f$ is not integrable in general.

We conclude our work with an explicit example, which extends to the octonions the main example of~\cite{ocs}.

\begin{example}
Consider the octonionic polynomial $x^2+xi$, which maps $\rr$ bijectively into a parabola $\gamma\subset\cc\subset\oo$. By direct computation,
\[f'_s(x)=t(x)+i,\quad f'_c(x)=2x+i\,.\]
In particular, $D_f=\emptyset$ and 
\[f'_c(\alpha+\beta J)f'_s(\alpha+\beta J)^{-1}=\frac{1+4\alpha^2+2\beta\langle J,i\rangle+4\alpha\beta J-2\beta J\times i}{1+4\alpha^2}\]
belongs to $\cc_J^\perp$ if, and only if, $\alpha=0$ and $\beta\langle J,i\rangle=-\frac12$. Thus,
\[N_f=-\frac{i}2+j\rr+k\rr+\ell\rr+\ell i \rr+\ell j \rr+\ell k \rr\,.\]
For all $x_0\in\oo\setminus N_f$, the preimage of $f(x_0)$ includes exactly two points: $x_0$ and $x_1=i-x_0$. On the other hand, $f$ is one-to-one from $N_f$ to $\Gamma:=f(N_f)$. Let us denote by $f^+$ the restriction of $f$ to $\oo^+=\{x\in\oo:\re(x)>0\}$ minus $\rr$ and by $f^-$the restriction of $f$ to $\oo^-=\{x\in\oo:\re(x)<0\}$ minus $\rr$. The functions $f^+$ and $f^-$ are injective and their ranges both equal
\[\oo\setminus(\gamma\cup\mathfrak{G})\,,\]
where $\mathfrak{G}$ is the image $f(\im(\oo))$ of the $7$-dimensional vector space $\im(\oo)$ bounding $\oo^+$ and $\oo^-$. The induced $\J^{f^+},\J^{f^-}$ are distinct almost-complex structures on $\oo\setminus(\gamma\cup\mathfrak{G})$. They are orthogonal because $f(\cc_i)\subseteq\cc_i$.
\end{example}



\end{document}